\documentclass[times,final]{elsarticle}
\usepackage{jcomp}
\usepackage{framed,multirow}
\usepackage{amssymb}
\usepackage{latexsym}
\usepackage{url}
\usepackage{amsthm}
\usepackage{xcolor}
\usepackage{amsmath}
\usepackage{comment}
\usepackage{mdframed}
\usepackage{enumitem}
\usepackage{graphicx}
\usepackage{caption}
\usepackage{subcaption}
\usepackage{mathtools}
\usepackage{algorithm2e}

\definecolor{newcolor}{rgb}
{.8,.349,.1}

\newtheorem*{notation}{Notation}
\newtheorem{theorem}{Theorem}[section]
\newtheorem{definition}{Definition}[section]
\newtheorem{proposition}{Proposition}[section]
\newtheorem{corollary}[theorem]{Corollary}
\newtheorem*{bem}{Bounded Entropy Method}
\newenvironment{bemthod}
  {\begin{mdframed}\begin{bem}}
  {\end{bem}\end{mdframed}} 
\newtheorem*{gen}{Generic Method}
\newenvironment{genthod}
{\begin{mdframed}\begin{gen}}
{\end{gen}\end{mdframed}}  

\SetKwComment{Comment}{/* }{ */}
\SetKwInput{KwData}{Input}
\SetKwInput{KwResult}{Output}
\newcommand{\algrule}[1][.2pt]{\par\vskip.3\baselineskip\hrule height #1\par\vskip.3\baselineskip}

\journal{Journal of Computational Physics}

\begin{document}

\verso{Bohyun Kim \textit{etal}}

\begin{frontmatter}

\title{A positivity-preserving numerical method for a thin liquid film on a vertical cylindrical fiber}

\author[1]{Bohyun \snm{Kim}\corref{cor1}}
\cortext[cor1]{Corresponding author:}
\ead{bohyunk@math.ucla.edu}
\author[2]{Hangjie \snm{Ji}}
\ead{hangjie_ji@ncsu.edu}
\author[1,3]{Andrea L. \snm{Bertozzi}}
\ead{bertozzi@g.ucla.edu}
\author[3]{Abolfazl \snm{Sadeghpour}}
\ead{abolfazlsad@ucla.edu}
\author[3]{Y. Sungtaek \snm{Ju}}
\ead{sungtaek.ju@ucla.edu}

\address[1]{Department of Mathematics, University of California, Los Angeles, CA 90095, USA}
\address[2]{Department of Mathematics, North Carolina State University, Raleigh, NC 27695, USA}
\address[3]{Department of Mechanical and Aerospace Engineering, University of California, Los Angeles, CA 90095, USA}

\received{TBD}
\finalform{TBD}
\accepted{TBD}
\availableonline{TBD}
\communicated{TBD}

\begin{abstract}
When a thin liquid film flows down on a vertical fiber, one can observe the complex and captivating interfacial dynamics of an unsteady flow. Such dynamics are applicable in various fluid experiments due to their high surface area-to-volume ratio. Recent studies verified that when the flow undergoes regime transitions, the magnitude of the film thickness changes dramatically, making numerical simulations challenging. In this paper, we present a computationally efficient numerical method that can maintain the positivity of the film thickness as well as conserve the volume of the fluid under the coarse mesh setting. A series of comparisons to laboratory experiments and previously proposed numerical methods supports the validity of our numerical method. We also prove that our method is second-order consistent in space and satisfies the entropy estimate.
\end{abstract}

\end{frontmatter}

%% main text
\section{Introduction} 
\label{section: Introduction}
%\textcolor{black}

Thin-film flows over fibers exhibit complex dynamical properties due to interplay among various forces, such as the surface tension, viscous force, gravity, and inertia force.  In the Rayleigh instability regime, an initially uniform flow quickly breaks up into regularly spaced beads, and forms traveling waves in the presence of gravity along the fiber direction \cite{quere1990thin,kalliadasis2011falling}. The beaded morphology creates an array of localized high-curvature regions that act as radial sinks, making it attractive for devices for heat and mass transfer along the liquid-gas interfaces \cite{sadeghpour2019water, gabbard2021asymmetric}.

These thin-film flows have applications in gas absorption \cite{uchiyama2003gas,grunig2012mass,chinju2000string}, heat exchange \cite{zeng2017experimental,zeng2018thermohydraulic}, microfluidics \cite{gilet2009digital},  desalination \cite{sadeghpour2019water}, and others. The wide variety of potential applications attracted theoretical studies over the last few decades \cite{quere1990thin,kalliadasis2011falling,craster2009dynamics, chang1999mechanism,duprat2009spatial,ruyer2008modelling,sadeghpour2017effects,quere1999fluid}.
\begin{figure}
    \centering
    \includegraphics[scale=0.17]{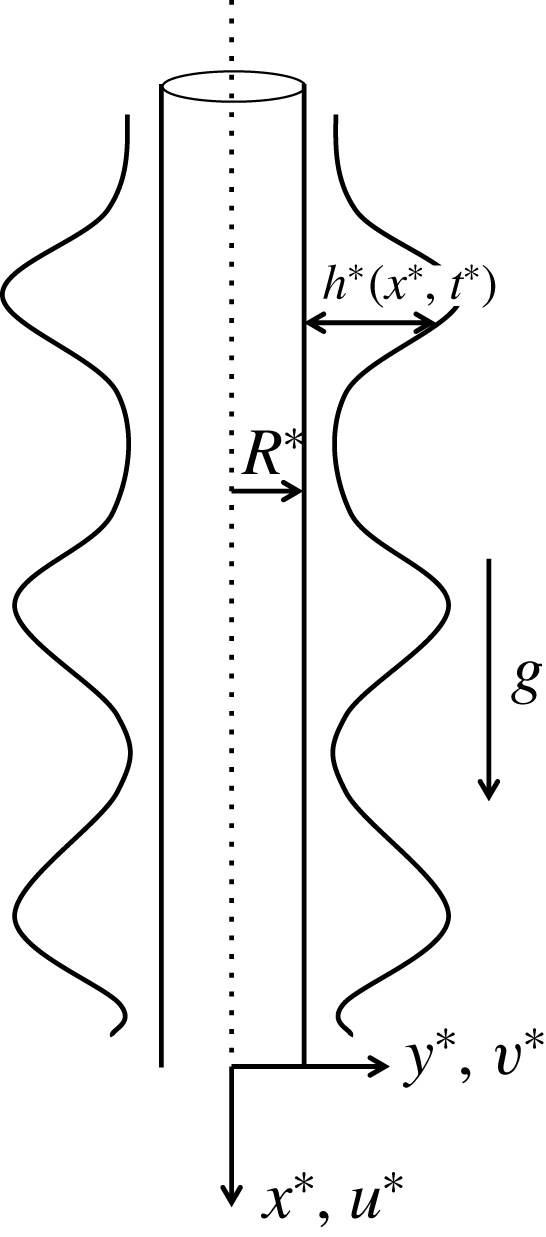}
    \caption{Illustration of a thin liquid film flowing down a vertical fiber. {$R^{*}$ represents the radius of a vertical cylinder, $h^{*}$ represents the thickness of the liquid film, $g$ represents the gravitational constant, and $u^*, v^*$ represent velocity components in the axial and radial direction. Dimensional variables are indicated with $*$ while we drop $*$ to represent corresponding dimensionless variables (see equation \eqref{generic fiber}). Ji et al., Dynamics of thin liquid films on vertical cylindrical fibres, Journal of Fluid Mechanics (2019), vol. 865, 303-327, reproduced with permission.} }
    \label{fig:my_label}
\end{figure}
The fundamental component determining the profile of the thin liquid film on a vertical fiber is surface tension, which has a stabilizing effect on the axial curvatures, and destabilizing effect on the azimuthal curvatures of the interface \cite{kliakhandler2001viscous}. In addition, other factors increasing the flow's complexity are the cylindrical geometry of the fiber and the gravitational force. Experimentally, interfacial instabilities of the flow have been studied over decades \cite{quere1990thin,quere1999fluid}. Kliakhandler et al. experimentally characterized the three distinct regimes of interfacial patterns $(a)$-$(c)$ \cite{kliakhandler2001viscous}. In this paper, we use the convention by Ji et al. \cite{ji2019dynamics} and call $(a)$-$(c)$ regimes convective, Rayleigh-Plateau, and isolated droplet regimes. The convective regime, observed when the flow rate is high, corresponds to the flow profile where irregular droplets collide with each other. The Rayleigh-Plateau regime corresponds to the flow profile, where beaded traveling waves propagate nearly constantly. The isolated droplet regime, observed when the flow rate is low, corresponds to the flow profile where small wavy patterns follow well-separated large droplets. The distinct dynamics of each regime and its transition is extensively studied, both theoretically and experimentally \cite{kalliadasis2011falling,quere1999fluid,ji2019dynamics,ruyer2009film,ruyer2012wavy,ji2021travelling}. 

In this paper, we consider reduced-order models of the Navier-Stokes equations incorporating linear and nonlinear effects of the flow. Li {\&} Chao \cite{li2020marangoni} summarize a few notable methods: the gradient expansion method \cite{ji2019dynamics,frenkel1992nonlinear,craster2006viscous,halpern2017slip}, the integral method \cite{sisoev2006film,trifonov1992steady}, the weighted residual method \cite{duprat2009spatial,ruyer2008modelling,ruyer2012wavy}, and the energy integral method \cite{novbari2009energy}. The models are often classified according to the size of the Reynolds number. For the low Reynolds number cases, the flow profile is approximated by the Stokes equations combined with the lubrication approximation \cite{ji2019dynamics,craster2006viscous}. For moderate Reynolds number cases, one incorporates inertial terms in the governing equation using the weighted residual boundary integral method \cite{duprat2009spatial,ruyer2008modelling}. Many of the models are verified against the experimental data \cite{duprat2009spatial,ruyer2008modelling}. 
For example, a recent study by Ji et al. shows a good agreement with experimental data by correctly predicting bead velocities, flow profiles, and regime transition bifurcation \cite{ji2019dynamics}.

A major challenge is that fiber coating equations are extremely difficult to solve both numerically and analytically. They are typically fourth-order degenerate nonlinear parabolic equations due to the surface tension in the dynamics. We consider the following model from \cite{ji2019dynamics}{:}  
\begin{equation}
\begin{split} \label{generic fiber}
    &\frac{\partial }{\partial t} \left(h+ \frac{\alpha}{2}h^2\right) + \frac{\partial}{\partial x}\mathcal{M}(h)+ \frac{\partial}{\partial x}\left[\mathcal{M}(h)\frac{\partial p}{\partial x} \right] = 0,  \\
    & \mathcal{M}(h) = O(h^n) ,\,\, p = \frac{\partial^2 h}{\partial x^2}-\mathcal{Z}(h).
\end{split}
\end{equation}

Equation~\eqref{generic fiber} is an evolution equation of {dimensionless} film thickness $h(x,t)$. From left to right, 

\begin{itemize}
    \item $\frac{\partial}{\partial t}(h+\frac{\alpha}{2}h^2)$ denotes the mass change over time where $\alpha = \mathcal{H}\slash{ R^{*}} \geq 0$ is the aspect ratio between the characteristic {radial} length scale of film thickness $\mathcal{H}$ to the fiber radius {$R^{*}$}. 

    \item $\mathcal{M}(h)$ is often referred to as the mobility function that describes the hydrodynamic interactions of the transverse waves. Many times, $\mathcal{M}(h) = O(h^n)$. For example,  setting {$\mathcal{M}(h) = h^3 $} corresponds to the no-slip boundary condition, and setting {$\mathcal{M}(h)=h^3 + \beta h^n$} for $n \in (0,3)$ corresponds to various Navier-slip conditions (cf. \cite{carro1995viscous}). The smoothness of $\mathcal{M}(h)$ near $h = 0$ determines the qualitative behavior of solutions at zero \cite{grun2001simulation}.

    \item The pressure $p$ consists of two terms - the linearized curvature $\frac{\partial^2 h}{\partial x^2}$, representing the streamwise surface tension, and the $\mathcal{Z}(h)$, representing other nonlinear pressure effects. $\mathcal{Z}(h)$ often contains a destabilizing surface tension term that arises from the azimuthal curvature but can also include other terms. %One can view $\mathcal{Z}(h)$ as a derivative of $\mathcal{U}(h)$, which is the potential energy.
\end{itemize}

Equation~\eqref{generic fiber} is considered state of the art for this problem because it quantitatively agrees with bead velocities, flow profiles, and regime transition bifurcations as compared to experiments. Previously, the model by Kliakhandler et al. \cite{kliakhandler2001viscous} incorporated fully nonlinear curvature to capture the qualitative behavior of the Rayleigh-Plateau and isolated droplet regime. Nevertheless, this model overestimated the beads’ velocity by 40\%. Craster \& Matar \cite{craster2006viscous} revisited this idea and presented an asymptotic model describing Rayleigh-Plateau and isolated droplet regime but again overestimated the bead velocity. Their model also identified the Rayleigh-Plateau regime to be transient rather than a stationary state. Duprat et al. \cite{duprat2007absolute}, and Smolka et al. \cite{smolka2008dynamics} further studied regime transitions but predicting the regime transitions remained challenging. Ji et al.’s film stabilization model (FSM) \cite{ji2019dynamics} improved the preceding models by incorporating a film stabilization term among generalized pressure terms. This stabilization term was inspired by the attractive part of the long-range apolar van der Waals forces, which are carefully studied for the well-wetting liquids \cite{reisfeld1992non,bonn2009wetting}. One can see that simulating such complex models is a delicate procedure. Thus, it is vital to have a robust numerical method for simulating complex spatiotemporal dynamics to predict flow profiles and regime transitions.  

The degeneracy of the mobility function $\mathcal{M}(h)$ and the complex nonlinear pressure terms $\mathcal{Z}(h)$ are two hurdles one needs to clear to construct a robust numerical method. First, the degeneracy of the mobility function presents a substantial challenge in numerically solving equation~\eqref{generic fiber} since the solution may lose regularity as $h \to 0$. Second, the nonlinear term $\mathcal{Z}(h)$ in pressure $p$ complicates the problem further since it is often relatively large in magnitude as $h \to 0$. As a result, the numerical method can suffer from instabilities as $h \to 0$. Therefore, keeping $h$ positive is not only crucial for the solution to be physically meaningful but also important for the solution to be accurate.  Fortunately, we found similarities between equation~\eqref{generic fiber} and many lubrication-type equations and realized we could view equation~\eqref{generic fiber} as a variant of a lubrication-type equation with generalized pressure \cite{grun2001simulation, oron1997long,bertozzi2001dewetting}.
\begin{align}
\label{lubrication}
    \frac{\partial h}{\partial t} + \frac{\partial}{\partial x}\left(\mathcal{M}(h)  \frac{\partial p}{\partial x}\right)  = 0 \quad p = \frac{\partial^2 h}{\partial x^2}  - \mathcal{Z}(h)&\quad \text{where} \quad f(h) \sim h^{n} \quad \text{as} \quad h \to 0.
\end{align}

One may see that setting $\alpha = 0$ and $\frac{\partial}{\partial x}\mathcal{M}(h) = 0$ in equation~\eqref{generic fiber} results in equation~\eqref{lubrication}. Setting $\alpha = 0$ would mean neglecting the effect of the fiber, and $\frac{\partial}{\partial x}\mathcal{M}(h) = 0$ would mean neglecting the advection effect by liquid traveling downward. Such experimental and theoretical settings are discussed in various studies devoted to the lubrication theory so that we can take advantage of them \cite{grun2001simulation,becker2003complex,lu2007diffuse,maki2010tear,grun2003convergence,bertozzi2001dewetting}. We know the solution of \eqref{lubrication} is smooth whenever the solution is positive but typically loses its regularity as the solution $h \to 0$ due to the degeneracy of the equation \cite{bertozzi1994singularities,bertozzi1996symmetric}. We also know that the nonlinear pressure terms often introduce a large numerical instability as $h \to 0$, making it challenging to maintain the positive numerical solution \cite{grun2001simulation,grun2003convergence}. Examples of fiber coating problems include $\mathcal{Z}(h)=-(\alpha/\epsilon)^2h$ in \cite{yu2013velocity}, assuming the thickness of the film is much smaller than the fiber radius ($\mathcal{H}\ll {R^{*}}$). Craster \& Matar \cite{craster2006viscous} used ${\mathcal{Z}}(h)=\frac{\alpha}{\eta(1+\alpha h)}$, assuming the film thickness comparable to the fiber radius  ($\alpha = O(1)$). Ji et al. \cite{ji2019dynamics} used that ${\mathcal{Z}}(h)=\frac{\alpha}{\eta(1+\alpha h)}-\frac{{A_H}}{h^{3}}$.
{The parameters $A_{H}$ and $\eta$ are discussed in more detail in Section 5}.  In both the Craster \& Matar{'s} model and Ji et al.{'s} model, we can expect numerical challenges when $h$ is small.  Indeed, we show in {S}ection \ref{section: comparison of numerical schemes} that the numerical method used in \cite{ji2019dynamics} can generate a false singularity as $h\to 0$ when the spatial grid size is underresolved. In other words, although the analytical solution of ~\eqref{generic fiber} is positive everywhere, the solution produced by a naive numerical method can produce negative values within some range of the solution when the grid size is underresolved. Such numerical methods can be quite difficult to extend to higher dimensions where grid refinement is computationally expensive. We also show that the negativity further prevents calculating the solution after the singularity. Thus, it is desirable to have a positivity-preserving numerical method that can perform well at different grid resolutions without spurious numerical singularities. 

Constructing {positivity-preserving} methods for partial differential equations (PDEs) is addressed in a wealth of literature yet most of them are limited to the {first-order} or second-order equations \cite{zhang2010positivity,zhang2011maximum,droniou2011construction,du2021maximum}. Equations above the second order have no maximum or comparison principles, and higher-order spatial derivatives make the numerical system extremely stiff. Numerical methods for {fourth-order} or higher-order equations with positivity-preserving properties have received far less attention. Early works include  \cite{grun2003convergence,barrett1998finite,zhornitskaya1999positivity,grun2000nonnegativity} and make use of entropy estimates to prove positivity. Some of the recent approaches use cut-off, or Lagrange multiplier methods which have a limitation in conserving mass or maintaining smoothness \cite{li2020arbitrarily,lu2013cutoff}. Here we introduce a convex-splitting method that preserves physical quantities like energy, entropy, and mass \cite{grun2001simulation,chen2019positivity,dong2019positivity,hu2020fully} which treats the stabilizing terms implicitly and the destabilizing terms explicitly. A few methods are unconditionally stable \cite{eyre1998unconditionally,vollmayr2003fast} which include the scalar auxiliary variable method by Huang et al. \cite{huang2022bound}. The applications of these methods are to solve Cahn-Hilliard or Hele-Shaw cell-type equations. 

This paper presents a positivity-preserving numerical scheme that works on a general family of lubrication-type equations on cylindrical geometries. Positivity-preserving numerical methods have not been studied in the context of fiber coating, especially in the regime that is most relevant to physical experiments. The structure of the paper follows. In {S}ection~\ref{section: PDE}, we prove properties that the PDE~\eqref{generic fiber} holds and discuss how the PDE imparts such properties to our numerical methods. In {S}ection~\ref{section: PPS formulation}, we introduce our numerical method and the state of art method used in Ji et al.~\cite{ji2019dynamics}. In {S}ection 4, we present proof of the positivity and the consistency of our method. Section~\ref{section: numerical simulation} contains numerical simulations of our methods. In particular, in {S}ection~\ref{section: comparison of numerical schemes}, we compare simulations of our method with simulations of the state of the art method while in {S}ection~\ref{section: comparison with experiment}, we compare simulations of our method with laboratory experimental data. We also demonstrate how to employ adaptive time stepping to efficiently implement our method in {S}ection~\ref{section: adaptive time stepping}. An example without any numerical singular behavior is presented in {S}ection~\ref{sec:nosing} whereas an example with a finite time numerical singular behavior is presented in {S}ection~\ref{sec: sing}. We also compare the CPU time of simulating our method and the state of the art method in {S}ection~\ref{section: computational efficiency}. Finally, in {S}ection~\ref{section: conclusion}, we conclude our paper with a few remarks and suggest future research directions. 

\section{Properties of the partial differential equation}
\label{section: PDE}
This section investigates two essential properties of the continuous fiber coating equation~\eqref{generic fiber}. We ensure that our numerical method preserves the discrete equivalent of the properties. We consider the following initial-boundary value problem:
\begin{equation*}
(P)\begin{dcases}  
&\frac{\partial }{\partial t} \left(h+ \frac{\alpha}{2}h^2\right) + \frac{\partial}{\partial x}\left[\mathcal{M}(h)\left(1+ \frac{\partial p}{\partial x} \right)\right] = 0 \,\, \text{in} \,\,  L_{T} = (0,L) \times (0,T) \subset \mathbb{R}^2,\\
& p =\frac{\partial^2 h}{\partial x^2} - \mathcal{Z}_{+}(h)-  \mathcal{Z}_{-}(h),\\
& [0,L]-\text{periodic boundary conditions},\\
& h(x,0) = h_0 (x) > 0. 
\end{dcases}
\end{equation*}

The main difference from previous equation \eqref{generic fiber} is that we split $\mathcal{Z}(h)$ into two parts: $\mathcal{Z}_{+}(h)$ and $\mathcal{Z}_{-}(h)$, where $\mathcal{Z'}_{+}(h) \geq 0$ and $\mathcal{Z'}_{-}(h) \leq 0$. Such splittings are not generally unique but useful in the design of stable numerical schemes. {Examples of convex-concave splitting can be found in many numerical works of Cahn-Hilliard or thin-film equations \cite{barrett1998finite, eyre1998unconditionally,doi:10.1137/0730084}.} An example is discussed in {S}ection~\ref{section: comparison of numerical schemes}. We assume periodic boundary conditions for simplicity and a positive initial condition to match the physical setting.

Here we assume that a smooth positive solution exists to the problem $(P)$. The existence of a solution to problems such as $(P)$ has been studied in depth \cite{ji2021travelling, bertozzi1994lubrication, bernis1990higher}. The general procedure is like this. First, one applies a regularization technique to problem $(P)$ to overcome the degeneracy and make the problem uniformly parabolic. The boundary condition can be extended to the whole line using a proper continuation technique such as the one suggested in \cite{solonnikov1965boundary}. The well-known parabolic Schauder estimates \cite{solonnikov1965boundary, desbrow_1971, 10.2307/24900462} guarantees a unique solution in a small time interval say, $L_{\sigma} = (0,L)\times (0,\sigma)$. In the end, the limit of the regularized solution results in a smooth, positive solution. We direct our readers to \cite{ji2021travelling, bernis1990higher} for the full derivation. We believe a similar derivation is possible through the canonical approach although continuation of solutions past the initial small time interval requires a priori bounds on certain norms. A full discussion of this problem is beyond the scope of this paper.

The key idea of developing a positivity-preserving numerical method is to formulate an entropy estimate for the continuous problem (P). Such an estimate guarantees the positivity of solutions in the continuous setting. Therefore, designing a numerical method that satisfies the discrete equivalent of the entropy estimate will result in a positivity-preserving numerical method. For our problem (P), we define entropy $G(h)$ so that its derivative $G'(h)$ satisfies 
\begin{equation*}
     G'(h) = \left(1+\alpha h\right)\int_{A}^{h}\frac{1}{\mathcal{M}(s)} \,ds ,\,\,\text{for some fixed } A> 0. 
\end{equation*}
We point out that the positivity proof for a continuous solution in {S}ection~\ref{section: PDE}, the definition of numerical methods in {S}ection~\ref{section: PPS formulation}, and the positivity proof for a discrete solution in {S}ection~\ref{section:positivity proof discrete solutions} do not explicitly involve the constant $A > 0$. In other words, $A$ is only involved in $G'(h)$ to ensure that it is well-defined. We claim that solutions to the problem (P) satisfy conservation of mass and an entropy estimate.

\begin{proposition}
\label{prop pde}
Suppose that there exists a solution $h \in C^4(L_{T})$ of $(P)$, where $L_{T}= [0,L)\times [0,T)$. Suppose we further assume
\begin{align*}
&\mathcal{M}(h) = O(h^n), \,\,\mathcal{M}(h) \geq 0,\\
   & \mathcal{Z}_{+}, \,\,\mathcal{Z}_{-} \in C^2(\mathbb{R}^{+}),\,\, \text{ and } \,\, \mathcal{Z'}_{+}(h) \geq 0,\,\, \mathcal{Z'}_{-}(h )\leq 0.
\end{align*}
Then, the solution $h$ satisfies the following two properties{:} 
\begin{align*}
&(I) \,\, \int_0^L h(x,T)+ \frac{\alpha}{2} h^2(x,T) \,dx = \int_0^L h(x,0)+ \frac{\alpha}{2} h^2(x,0) \, dx \,\, (\text{Conservation of mass}),\\
&(II) \int_0^L G(h(x,T))\, dx \leq \int_0^L G(h(x,0))\, dx +\int_{L_{T}}\left(\frac{\mathcal{Z}_{-}(h)}{2}\right)^2dx dt \,\,\text{(Entropy estimate}).
\end{align*}
\end{proposition}

\begin{proof}
The conservation of mass (I) is achieved by integrating the problem $(P)$ on ${L_{T}}${:} 
\begin{align*}
    &\int_{L_{T}} \frac{\partial }{\partial t} \left(h+ \frac{\alpha}{2}h^2\right) dx dt   = - \int_{L_{T}} \frac{\partial}{\partial x}\left[\mathcal{M}(h)\left(1+ \frac{\partial p}{\partial x} \right)\right] dx dt \\
    &\implies \int_0^L  \left(h(x,T)+ \frac{\alpha}{2}h^2(x,T)\right) dx - \int_0^L  \left(h(x,0)+ \frac{\alpha}{2}h^2(x,0)\right) dx = 0.
\end{align*}
Note that the periodic boundary condition removes the complex expression surrounded by $\frac{\partial}{\partial x}\left[... \right]$ on the {right}-hand side of the equality in the first line.

The entropy estimate (II) is achieved by directly calculating the time derivative of $G(h)${:}
\begin{align*}
    \frac{d}{dt}\int_0^L G(h) dx &= \int_0^L G'(h) h_t dx \\
    &= \int_0^L \left\{(1+\alpha h) h_t \int_A^{h} \frac{1}{\mathcal{M}(s)} ds\right\} dx\\
    &=   -\int_0^L\left\{\frac{\partial}{\partial x}\left[\mathcal{M}(h)\left(1+ \frac{\partial p}{\partial x} \right)\right] \int_A^{h} \frac{1}{\mathcal{M}(s)} ds\right\} dx\\
    &=   \int_0^Lh_x\left(1+ \frac{\partial p}{\partial x} \right) dx.
\end{align*}
The equalities are justified by the integration by parts. Note that the periodic boundary plays a crucial role in simplifying expressions on the boundary. We use the definition $p = h_{xx}-\mathcal{Z}(h) =h_{xx} -\mathcal{Z}_{+}(h) -\mathcal{Z}_{-}(h) $ to continue our calculation{:}
\begin{align*}
     \frac{d}{dt}\int_0^L G(h) dx &= \int_0^L h_x dx + \int_0^L h_{x}\frac{\partial }{\partial x}\left(h_{xx}-\mathcal{Z}(h)\right) dx\\
    & = -\int_0^L h^2_{xx}+\int_0^L h_{xx}\mathcal{Z}_{-}(h) dx -\int_0^L h^2_x\mathcal{Z'}_{+}(h)  dx\\
    & = - \int_0^L \left(h_{xx}- \frac{Z_{-}(h)}{2}\right)^2 dx + \int_0^L\left(\frac{Z_{-}(h)}{2}\right)^2 dx - \int_0^L h^2_x\mathcal{Z'}_{+}(h)  dx\\
    & \leq - \int_0^L \left(h_{xx}- \frac{Z_{-}(h)}{2}\right)^2 dx + \int_0^L\left(\frac{Z_{-}(h)}{2}\right)^2 dx.
\end{align*}

Again, the periodic boundary is crucial in eliminating $\int_0^L h_x dx$ in the first line. We simplify the expression by completing the square on the third line. We obtain the inequality in the last line because $\mathcal{Z'}_{+}(h) \geq 0$. Integrating over time gives us
\begin{equation*}
     \int_0^L G(h(x,T))\, dx+ \int_{L_{T}} \left(h_{xx}- \frac{\mathcal{Z}_{-}(h)}{2}\right)^2dx dt\leq \int_0^L G(h(x,0))\, dx +\int_{L_{T}}\left(\frac{\mathcal{Z}_{-}(h)}{2}\right)^2dx dt.
\end{equation*}
Finally, one can drop the second term on the left side of the inequality since it is nonnegative.
\end{proof}

The above properties allow us to create a positivity-preserving numerical method due to the entropy estimate. Lubrication-type equations are well-known to satisfy entropy-dissipating properties. Bernis et al. recognized the significance of the entropy dissipation property in third-order or higher degenerate parabolic equations and used it to prove the nonnegativity of weak solutions with sufficiently high degeneracy in one space dimension \cite{bernis1990higher}. They also proved that the solution is unique and strictly positive if the mobility order $n \geq 4$. Following their work, several articles regarding lubrication-type equations discussed the importance of entropy estimates in numerical and analytical contexts {\cite{grun2001simulation, bertozzi1994singularities, barrett1998finite, zhornitskaya1999positivity,grun2000nonnegativity,bertozzi1994lubrication,
bertozzi1996lubrication, nonnegativeBeretta, doi:10.1137/S0036141096306170}}. These ideas have largely been lacking in the fiber coating problem, except for the entropy analysis done by Ji et al. \cite{ji2021travelling}, which proves the existence of a {generalized} nonnegative weak solution of a fiber-coating model with fully nonlinear curvature terms {on a periodic domain}. In this paper, we use these ideas to develop a positivity-preserving numerical {method}.

\section{Positivity-preserving Finite difference method}
\label{section: PPS formulation}
In this section, we present a continuous time and discrete in space positivity-preserving finite difference method, the Bounded Entropy Method (BEM), and compare it to the current state of the art method General Method (GM) used in fiber coating models \cite{ji2019dynamics}. Our method is second-order accurate in space while preserving the positivity of a numerical solution at each time. Our method is motivated by prior work by Zhornitskaya {\& Bertozzi} \cite{zhornitskaya1999positivity} and {Gr{\"u}n \& Rumpf} \cite{grun2001simulation} for a simple lubrication{-type} model without the geometry and physics of fiber coating. Before introducing our method, we define the following notation.

\begin{notation}
Suppose we divide our domain $[0,L]$ into $N$ equally spaced grids of size $\Delta x = L/N$. Let $u_{i}(t)$ be a solution of a numerical method that is continuous in time and discrete in space at time $t$ and on grid $i$. Define the forward difference in space and the backward difference in space as
\begin{equation*}
 u_{i,x} = \frac{u_{i+1}(t) - u_{i}(t)}{\Delta x}, \quad  u_{i,\bar{x}} = \frac{u_{i}(t) - u_{i-1}(t)}{\Delta x}. 
    \end{equation*}
    Respectively, higher-order differences in space can be defined as
    \begin{equation*}
     u_{i,\bar{x}x} = \frac{u_{i+1,\bar{x}} - u_{i, \bar{x}}}{\Delta x} , \quad u_{i,\bar{x}x\bar{x}} = \frac{u_{i, \bar{x}x} - u_{i-1, \bar{x}x}}{\Delta x}.
     \end{equation*}

\end{notation}
As we highlight the importance of the entropy $G(h)$ in designing a positivity-preserving method in {S}ection~\ref{section: PDE}, the discretized mobility $\mathcal{M}(h)$ is the key factor that determines the qualitative behavior of the solutions near zero. We define the discrete mobility function $m(s_1,s_2)$ according to Definition~\ref{mobility discretization}.

\begin{definition}
[Discretization of Mobility]
\label{mobility discretization}
The mobility term $\mathcal{M}(s)$ in the problem $(P)$ is discretized to satisfy the following criteria \cite{zhornitskaya1999positivity}{:}
\begin{enumerate}[label=(\alph*)]
\item $m(s,s) = \mathcal{M}(s)$,
\item $m(s_1,s_2) = m(s_2,s_1)$,
\item $m(s_1,s_2) \in C^{4}((0,\infty)\times (0,\infty)) \cap C([0,\infty]\times[0,\infty])$,
\item $\forall \delta >0,$ there exists $\gamma >0 $ such that $s_1,s_2 >\delta \implies m(s_1,s_2) \geq \gamma > 0$.
\end{enumerate}    
\end{definition}The above definition of $m(s_1,s_2)$ is symmetric and continuously differentiable everywhere except possibly at $0$. Condition (d) allows the $m(s_1,s_2)$ to be degenerate if one of the arguments $h \to 0$ but guarantees positivity if both of the arguments are greater than 0.
Our positivity-preserving finite difference method, the Bounded Entropy Method (BEM), presented below, satisfies Definition~\ref{mobility discretization}.

\begin{bemthod}[BEM]
\label{pps}
 The finite difference discretization of the problem $(P)$ with continuous time is written by the following equations{:}
\begin{equation} \label{BEM}
\begin{split}
 & (1+\alpha u_{i}) \frac{d u_i}{dt}+  [m(u_{i-1},u_{i})(1+p_{i,\bar{x}})]_{x} = 0, \quad p _{i} = u_{i, \bar{x}x} - \mathcal{Z}_{+}(u_{i})- \mathcal{Z}_{-}(u_{i}),\\
  & u_i(0) = u_0(i\Delta x), \, i = 0,1,2 \cdots N,\\
 & m(s_1,s_2) = \begin{cases}
           \mathcal{M}(s_1)  & \text{if}\,\, s_1 = s_2,\\
           (s_2-s_1)/\int_{s_1}^{s_2}\frac{1}{\mathcal{M}(s)}ds & \text{if}\,\, s_1 \ne s_2.
       \end{cases}
\end{split}
\end{equation}
\end{bemthod}

In {S}ection~\ref{section:positivity proof discrete solutions}, we show that the above discretization of $\mathcal{M}(h)$ in BEM~\eqref{pps} guarantees a discrete equivalent of the conservation of mass (I) and the entropy estimate (II). We also write the numerical method of Ji et al. \cite{ji2019dynamics} as the following, which we refer to as Generic Method (GM).

\begin{genthod}[GM]
The finite difference discretization of the problem $(P)$ with continuous time is written by the following equations{:}
\begin{equation}
\label{GM}
\begin{split}
    &(1+\alpha u_{i}) \frac{d u_i}{dt}+  [m(u_{i-1},u_{i})(1+p_{i,\bar{x}})]_{x} = 0, \quad p _{i} = u_{i, \bar{x}x} - \mathcal{Z}_{+}(u_{i})- \mathcal{Z}_{-}(u_{i}),\\
    & u_i(0) = u_0(i\Delta x), \, i = 0,1,2 \cdots N,
\end{split}
\end{equation}
where $m(s_1,s_2)$ satisfies Definition \ref{mobility discretization}.

\end{genthod}
As an example of $m(s_1,s_2)$ used in GM~\eqref{GM}, one can let $m(s_1,s_2) = \mathcal{M}(0.5(s_1+s_2))$ or $m(s_1,s_2) = 0.5(\mathcal{M}(s_1)+\mathcal{M}(s_2))$, where either one estimates the mobility at the midpoint. Note that $m(s_1,s_2)$ in BEM~\eqref{pps} and GM~\eqref{GM} uses center-difference, allowing the numerical method to conserve flux at each time step. Together with second-order consistency, both numerical methods are ``shock capturing,” which is a desirable property to have in conservation law type of equations \cite{leveque1992numerical}. In the following section, we show that BEM~\eqref{pps} satisfies the conservation of mass and entropy estimate, which allows us to prove the positivity of the numerical method. 

\section{Positivity of Numerical solutions}
\label{section:positivity proof discrete solutions}
In the previous section, we claim that $m(s_1,s_2)$ in BEM~\eqref{pps} satisfies a discrete equivalent of the conservation of mass and the entropy estimates discussed in {S}ection~\ref{section: PDE}. In this section, we prove our claim through Proposition~\ref{prop numerical} and explain how such discretizations preserve the positivity of BEM~\eqref{pps} through Theorem~\ref{pps bem}. Our method is inherently more complex than entropy dissipating schemes for traditional lubrication-type equations because of three reasons. First, the time derivative of \eqref{generic fiber} involves the geometry of the cylindrical fiber $\frac{\alpha}{2}h^2$. Second, a nonlinear advection $\frac{\partial }{\partial x}\mathcal{M}(h)$ is incorporated. Lastly, nonlinear pressure $p$ entails { $\mathcal{Z}(h) = \mathcal{Z}_{+}(h)+ \mathcal{Z}_{-}(h)$}. The coupled entropy estimate expression in Proposition 2.1 is consequently more complicated than ``entropy dissipation", which is the case for the conventional lubrication-type equations. The following proposition is a discrete analog of Proposition 2.1.

\begin{proposition}
\label{prop numerical}
    Suppose $u_{i}(t)$ is a solution of the BEM~\eqref{pps} at time $t$ and $i$-th grid in space. Suppose we further assume
\begin{align*}
    & \mathcal{M}(h) = O(h^n),\, \mathcal{M}(h) \geq 0,\\
    & \mathcal{Z}_{+},\, \mathcal{Z}_{-} \in C^2(\mathbb{R}^{+}),\,\,\text{ and }\,\,\mathcal{Z'}_{+}(h) \geq 0,\, \mathcal{Z'}_{-}(h )\leq 0,\\
    & G'(h) = \left(1+\alpha h\right)\int_{A}^{h}\frac{1}{\mathcal{M}(s)} \,ds ,\,\,\text{for some fixed  } A > 0. 
\end{align*}

Then, $u_{i}(t)$  satisfies the following two properties given $T > 0$;
\begin{align*}
&(I) \,\, \sum_{i} \left(u_i(T)+ \frac{\alpha}{2} u_i(T)^2\right) \,\Delta x = \sum_i \left(u_i(0)+ \frac{\alpha}{2} u_i(0)^2\right) \, \Delta x \,\, (\text{Discrete conservation of mass}),\\
&(II) \sum_{i} G(u_{i}(T))\Delta x \leq \sum_{i}G(u_{i}(0))\Delta x + \int_0^T  \sum_{i}\left(\frac{\mathcal{Z}_{-}(u_i({t}))}{2}\right)^2\Delta x d{t} \,\,(\text{Discrete entropy estimate}).
\end{align*}
\end{proposition}

\begin{proof}
The proof of the statements is very similar to the proof of Proposition 2.1. The only difference is that we multiply by $\Delta x$ and sum over $i =  0,1,2...\,N$ instead of integrating over space. Discrete conservation of mass (I) is achieved by integrating the first line of \eqref{BEM} by time and summing over $i =  0,1,2...\,N${:} 
\begin{align*}
    & \int_0^T \sum_i   (1+\alpha u_{i})\frac{d u_i}{dt} \Delta x{dt} = -\int_0^T \sum_i [m(u_{i-1},u_{i})(1+p_{i,\bar{x}})]_{x} \Delta x {dt} \\
   &\implies \sum_i  \left(u_i(T)+ \frac{\alpha}{2}u_i(T)^2\right) \Delta x -  \sum_i  \left(u_i(0)+ \frac{\alpha}{2}u_i(0)^2\right) \Delta x = 0.
\end{align*}
As we saw in the continuous case, the periodic boundary condition removes the expression surrounded by $[ ...]_{x}$.

The discrete entropy estimate (II) is achieved by direct calculation.
\begin{align*}
    \frac{d}{dt}\sum_{i} G(u_{i}) \Delta x &= \sum_{i} G'(u_{i}) \frac{d u_{i}}{dt} \Delta x \\
    &= -\sum_{i} \int_A^{u_i}\frac{1}{\mathcal{M}(s)}ds[{m(u_{i-1},u_{i})}(1+p_{i,\bar{x}})]_{x} \Delta x\\
    &= \sum_{i}\frac{1}{\Delta x}\left(\int_{u_{i-1}}^{u_i} \frac{1}{\mathcal{M}(s)}ds\right){m(u_{i-1},u_{i})}(1+p_{i,\bar{x}}) \Delta x\\
    & = \sum_{i}u_{i,\bar{x}}(1+p_{i,\bar{x}}) \Delta x \,\,\, \\
    & = \sum_{i} \left\{ -(u_{i,\bar{x}x})^2 - u_{i,\bar{x}} [\mathcal{Z}_{+}(u_{i})]_{\bar{x}}+ u_{i,\bar{x}x} \mathcal{Z}_{-}(u_{i}) \right\} \Delta x  \\
    & \leq  - \sum_{i} \left(u_{i,\bar{x}x}- \frac{\mathcal{Z}_{-}(u_i)}{2}\right)^2\Delta x +\sum_{i}\left(\frac{\mathcal{Z}_{-}(u_i)}{2}\right)^2\Delta x. 
\end{align*}

Until the 4th line, the equalities are justified by integration by parts. Note that the periodic boundary plays a crucial role in simplifying expressions on the boundary and eliminating $\sum_{i} u_{i,\bar{x}}\Delta x$ in the 4th line. We obtain the inequality in the last line after completing the square and using the fact that $\mathcal{Z'}_{+} \geq 0$. From the inequality, one integrates over time from $0$ to $T$.
\begin{align*}
    &\sum_{i} G(u_{i}(T))\Delta x + \int_0^T \sum_{i}\left(u_{i,\bar{x}x} ({t})- \frac{\mathcal{Z}_{-}(u_i({t}))}{2}\right)^2\Delta x d{t} \leq \sum_{i}G(u_{i}(0))\Delta x + \int_0^{{T}}  \sum_{i}\left(\frac{\mathcal{Z}_{-}(u_i({t}))}{2}\right)^2\Delta x d{t}
    \end{align*}
Finally, one can drop the second term on the left side since it is nonnegative and the desired entropy estimate is achieved. 
\end{proof}

We have two versions of theorems on the positivity: (a) a priori bound - depending on $\Delta x$ and (b) a posteriori bound assuming a uniform Lipschitz condition on the numerical solution. We note that the solution is observed to have a uniform Lipschitz bound in all of our numerical simulations. Thus, the uniform Lipschitz assumption is observed numerically and thus can be used in an {\em a posteriori} argument. We leave proving the smoothness of PDE, such as establishing a uniform Lipschitz bound, as future work.

\begin{theorem}{(Positivity of BEM)}
\label{pps bem}
Suppose we have the same assumptions as Proposition 4.1. We further assume that $(Z_{-}(s))^2 \leq C_1$ for any $s \geq 0$ and the initial data $u_i(0)> 0$. Then, the solution of BEM~\eqref{BEM} at time $T >0 $, $u_i(T)$, satisfies the following conditions; 

 \begin{enumerate}[label=(\alph*)]
     \item if $n \geq 2$, there exists $\delta$ such that $u_i(T) \geq \delta (\Delta x) > 0$ for all $i$,
     \item {if $n > 2$ and $u_i(t)$ is uniformly Lipchitz on $[0,T]$, there is a posteriori lower bound $\delta$ independent of $\Delta x$ such that $u_i(T) \geq \delta > 0$. i.e. One assumes $|u_{i}(t) -u_{j}(t)| \leq C_L|(i-j)\Delta x|$ for some $C_L>0$ and for $\forall i,j$, $\forall 0\leq t\leq T$.}  
 \end{enumerate}
 
\end{theorem}
\begin{proof}
Notice that we assume that $\mathcal{M}(h) =  O(h^n)$ and consider cases where $n \geq 2$. Thus, for the sake of simplicity, we take $\mathcal{M}(h) = h^n$ throughout the proof. More general cases can be proved similarly.  Let us first prove statement (a). The given assumptions allow us to use the discrete entropy estimate $(II)$ from Proposition \ref{prop numerical}. First, we claim that $\sum_{i} G(u_{i}(T)) \Delta x \leq C$ for a fixed constant $C$ as any $u_i(T) \to 0$. Since we take $\mathcal{M}(h) = h^n$, we can explicitly calculate 
\begin{equation*}
G(h) =
\begin{cases}   
& -\ln h + O(h)+ O(1) \quad \text{if  } n = 2, \\
%& -\ln h +\left(\frac{1}{A}-\alpha\right)h+ \frac{\alpha}{2A} h^2+O(1) \quad \text{if } n = 2, \\
& \frac{1}{(n-1)(n-2)}h^{-(n-2)}  + O(h^{3-n}) + O(1) \quad \text{if  } 2 < n < 3,\\
& \frac{1}{2h} -\frac{\alpha}{2} \ln h + O(h) + O(1) \quad  \text{if  } n = 3,\\
& \frac{1}{(n-1)(n-2)}h^{-(n-2)} + \frac{\alpha}{(n-1)(n-3)} h^{-(n-3)} + O(h) + O(1)\quad  \text{if  } n > 3.
\end{cases}
\end{equation*}

Here, the choice of $A$ only affects the coefficients of the higher-order terms but not the leading-order term. Each $G(u_i(0))$ is also well defined because we have fixed initial data $u_i(0) > 0$. This leads us to conclude
\begin{equation*}
    \sum_{i}G(u_{i}(0))\Delta x \leq  C_0,\,\, \text{for some constant} \,\, C_0.
\end{equation*}
We also assume {$(Z_{-}(s))^2 \leq C_1$ for any $s \geq 0$} so 
\begin{equation*}
    \int_0^T  \sum_{i}\left(\frac{\mathcal{Z}_{-}(u_i({t}))}{2}\right)^2\Delta x {dt} \leq C_2T, \,\,\text{for some constant}\,\, C_2,\,\, \text{as any } u_i(T) \to 0.
\end{equation*}
Hence, we get
\begin{align*}
    &\sum_{i} G(u_{i}(T))\Delta x\leq \sum_{i}G(u_{i}(0))\Delta x + \int_0^T  \sum_{i}\left(\frac{\mathcal{Z}_{-}(u_i({t}))}{2}\right)^2\Delta x d{t} \leq C_0 + C_2T \leq C.
\end{align*} 

Next, we show that $\delta(T) = \min_i u_i(T) \geq 0$ using the boundedness of $\sum_{i} G(u_{i}(T))\Delta x$. Notice that each leading-order term of $G(\delta)$ is positive as $\delta \to 0$, up to constant differences.
\begin{equation*}
G(\delta) =
\begin{cases}   
& -\ln \delta + O(\delta)+ O(1) \quad \text{if  } n = 2, \\
& \frac{1}{(n-1)(n-2)}\delta^{-(n-2)}  + O(\delta^{3-n}) + O(1) \quad \text{if  } 2 < n < 3,\\
& \frac{1}{2\delta} -\frac{\alpha}{2} \ln \delta + O(\delta) + O(1) \quad  \text{if  } n = 3,\\
& \frac{1}{(n-1)(n-2)}\delta^{-(n-2)} + \frac{\alpha}{(n-1)(n-3)} \delta^{-(n-3)} + O(\delta) + O(1)\quad  \text{if  } n > 3,
\end{cases}
\end{equation*}
Thus, $\delta \to 0$ implies $G(\delta) \to +\infty$, which contradicts $\sum_{i}G(u_i(T))\Delta x \leq C$. Hence, we achieve $\min_i u_i(T) = \delta > 0$.

To prove (b), we use $\sum_{i} G(u_{i}(T))\Delta x \leq C $  as well. From part (a), we have nonnegativity of $u_i(T)$ so 
\begin{align*}
 G(u_i(T)) = \int_B^{u_i} (1+\alpha v) \int_A^{v} \frac{1}{\mathcal{M}(s)}dsdv + O(1)\geq  \int_B^{u_i} \int_A^{v} \frac{1}{\mathcal{M}(s)}dsdv + O(1), \,\, \text{for some} \,\, B > 0. 
\end{align*}
Therefore, 
\begin{align*}
     C \geq \sum_i G(u_i(T))\Delta x \geq \sum_i\int_B^{u_i} \int_A^{v} \frac{1}{\mathcal{M}(s)}dsdv \Delta x + O(1) \geq \sum_{i} \int_B^{u_i} \int_A^{v} \frac{1}{s^n}dsdv \Delta x + O(1)  =  \sum_i u^{2-n}\Delta x+O(1). 
\end{align*}
Suppose $\delta(T) = \min_i u_i(T)$ occurs at $i^{*}$. Due to the uniform Lipschitzness, $u_i \leq \delta + C_L |(i^{*}-i)\Delta x|,\, \forall i$ {so} 
\begin{align*}
    \Tilde{C} \geq  & \sum_i\frac{1}{u^{n-2}_i}\Delta x \geq \sum_i\frac{\Delta x}{(\delta + C_L|(i-i^*)\Delta x|)^{n-2}} \geq \sum_i\frac{\Delta x}{(\delta + C_L(i\Delta x))^{n-2}} \\
    & \geq \int_0^L \frac{dx}{(\delta + C_L x)^{n-2}} \geq \frac{1}{C_L \delta^{n-1}}\int_0^{LC_L\slash\delta}\frac{ds}{(1+s)^{n-2}}. 
\end{align*}
If $ \frac{L C_L}{\delta} \leq 1 \implies \delta \geq LC_L$ so we have lower bound for $\delta$ independent of $\Delta x$. In the case when $\frac{L C_L}{\delta} \geq 1$,
\begin{align*}
    &\Tilde{C}  \geq \frac{1}{C_L \delta^{n-1}}\int_0^{1} \frac{ds}{(1+s)^{{n-2}}} = \frac{C'}{\delta^{n-1}}\\
    &\implies \delta \geq \left(\frac{C'}{\Tilde{C} }\right)^{1\slash n-1}. 
\end{align*}
\end{proof}

\begin{corollary}
Continuous time, discrete space, numerical solutions of the Craster-Matar model (CM) \cite{craster2006viscous} and the Film Stabilization Model (FSM) \cite{ji2019dynamics} are positive at any time $T>0$ and grid point $i$ if we use the BEM~\eqref{pps}. 

\end{corollary}

\begin{proof}
For both cases, the same mobility function $\mathcal{M}(h)$ is used, but different $\mathcal{Z}(h)$ is used{:} 
\begin{align*}
&\mathcal{M}(h) = \frac{h^3}{3}\frac{\phi(\alpha h)}{\phi(\alpha)} +\frac{h^2(\alpha h + 2)^2 \lambda}{4\phi(\alpha)},\\
&\phi(x) = \frac{3}{16x^3} \left[(1+x)^4(4\ln(1+x)-3) + 4(1+x)^2 -1\right],\\
&\mathcal{Z}_{CM}(h) = \mathcal{Z}_{CM-}(h)=\frac{\alpha}{\eta(1+\alpha h)},\\
&\mathcal{Z}_{FSM}(h) = \mathcal{Z}_{FSM+}+\mathcal{Z}_{FSM-} = -\frac{A_{H}}{h^3}+ \frac{\alpha}{\eta(1+\alpha h)},
\end{align*}
{for $\alpha,\eta,A_H >0$.} We prove that the assumptions for Theorem~\ref{pps bem} are satisfied by showing that $\mathcal{M}(h)= O(h^2)$ as $h \to 0$ and $(\mathcal{Z}_{-}(s))^2 \leq \left(\frac{\alpha}{\eta}\right)^2$. To simplify the calculation, let $y =\alpha h$. Then, {we achieve}
\begin{align*}
{\mathcal{M}(h)} &= \frac{1}{16\alpha^3\phi(\alpha)}\left[(y+1)^4 (4\ln(y+1)-3) + 4(y +1)^2 -1+4\lambda\alpha y^2(y+2)^2\right]\\
&  =  \frac{1}{C}\left[A_4 y^4 + A_3 y^3 + A_2 y^2 + A_1 y + A_0\right], 
\end{align*}
where
\begin{align*}
    & A_4 = 4\alpha\lambda + 4\ln(y+1)-3, \,\,A_3 = 16\alpha\lambda + 16\ln(y+1) -12,\,\, A_2 = 16\alpha\lambda + 24\ln(y+1)-14,\\
    & A_1 = 16\ln(y+1)-4,\,\, A_0= 4\ln(y+1).
\end{align*}

As $y \to 0$, $\ln(y+1) = O(y)$. Thus, 
\begin{align*}
{\mathcal{M}(h)} = O(y^2) + \frac{1}{C}[A_1 y + A_0] = O(y^2)+16y^2 -4y + 4y = O(y^2) = {O(h^2)}.
\end{align*}
Finally, for any $s \geq 0$,
\begin{align*}
     \mathcal{Z}_{-}(s) = \frac{\alpha}{\eta(1+\alpha s)}
     \leq \frac{\alpha}{\eta}.
\end{align*}
To finish the proof, we apply Theorem~\ref{pps bem} and see that the numerical solutions of both CM and FSM are positive.
\end{proof}

\begin{theorem}[Consistency]
GM~\eqref{GM} and BEM~\eqref{pps} are second-order consistent in space. That is, given a smooth solution $u(x,t)$ of the problem $(P)$, a local truncation error $\tau_i(t)$ is 
$O(\Delta x^2)$, {w}here
\begin{equation*}
    \tau_i(t) = (1+\alpha u_i){\frac{du_{i}}{dt}} + [m(u_{i-1},u_{i})(1+p_{i,\bar{x}})]_{x}.
\end{equation*}
\end{theorem}
\begin{proof}
 Let us denote $u_{i} = u(i\Delta x,t)$ to simplify the notation. First, note that both GM and BEM have very similar formulations and satisfy Definition \ref{mobility discretization}. Thus, we can use an approach similar to \cite{zhornitskaya1999positivity}. After Taylor expansion,
\begin{align*}
    m(s_1,s_2) &= m(s+\Delta s, s-\Delta s) = m(s,s) +\frac{\partial m}{\partial s_1}(s,s)\Delta s - \frac{\partial m}{\partial s_2} (s,s)\Delta s + \beta(s) \Delta s^2 + O(\Delta s^2)\\
    &=\mathcal{M}(s) + \beta(s) \Delta s^2 + O(\Delta s^2),
\end{align*}
where $s = \frac{s_1 + s_2}{2}, \Delta s = \frac{s_1-s_2}{2}$, and
\begin{equation*}
    \beta(s) = \frac{1}{2} \left(\frac{\partial^2 m(s,s)}{\partial s_1^2} -2 \frac{\partial^2 m(s,s)}{\partial s_1 \partial s_2} + \frac{\partial^2 m(s,s)}{\partial s_2 ^2}\right){.}
\end{equation*}
We cancel out  $O(\Delta s)$ terms by using the symmetry of $m(s_1,s_2)$, according to (b) from Definition \ref{mobility discretization}.
We also obtain
\begin{align*}
    & p_{i,\bar{x}} =  u_{i, \bar{x}x\bar{x}} - \left[\mathcal{Z}(u_{i})\right]_{\bar{x}}{,}\\
    &u_{i, \bar{x}x\bar{x}} = \frac{u_{i+1}-3u_{i}+3u_{-1}-u_{i-2}}{\Delta x^3} = u^{(3)}_{i-\frac{1}{2}} + \alpha (x_{i-\frac{1}{2}})\Delta x^2 + O(\Delta x^4){,}\\
    &[\mathcal{Z}(u_i)]_{\bar{x}} = \mathcal{Z'}(u_{i-\frac{1}{2}})\frac{u_i - u_{i-1}}{\Delta x}+  \mathcal{Z''}(u_{i-\frac{1}{2}})\frac{(u_i - u_{i-\frac{1}{2}})^2 -(u_{i-1} - u_{i-\frac{1}{2}})^2 }{2\Delta x} + O(\Delta x^2) + O(\Delta x^4)  \\
    & \quad = \mathcal{Z'}(u_{i-\frac{1}{2}})\left[{u'}_{i-\frac{1}{2}} + \frac{\Delta x^2}{24}u^{(3)}_{i-\frac{1}{2}} + O(\Delta x^4)\right] +  \mathcal{Z''}(u_{i-1/2})\left[\frac{\Delta x^2}{8} {u'}_{i-\frac{1}{2}} {u''}_{i-\frac{1}{2}} + O(\Delta x^4)\right].
\end{align*} 
After a simplification, {we achieve}
\begin{align*}
    p_{i,\bar{x}} = u^{(3)}_{i-\frac{1}{2}}+ \mathcal{Z'}(u_{i-\frac{1}{2}}){u'}_{i-\frac{1}{2}}+\gamma(x_{i-\frac{1}{2}})\Delta x^2 + O(\Delta x^4)
\end{align*}
for some smooth function $\gamma(x)$. 

As a result, \begin{align*}
    &[m(u_{i-1},u_{i})(1+p_{i,\bar{x}})]_{x}= \frac{1}{\Delta x} \left[m(u_{i},u_{i+1})(1+p_{i+1,\bar{x}})-m(u_{i-1},u_{i})(1+p_{i,\bar{x}})\right]\\
    &= \frac{1}{\Delta x}\left\{\mathcal{M}\left(\frac{u_i+u_{i+1}}{2}\right) + \beta\left(\frac{u_i+ u_{i+1}}{2}\right)\left(\frac{u_{i+1}-u_{i}}{2}\right)^2 + O(\Delta x^3)\right\}\left\{1+ u^{(3)}_{i+\frac{1}{2}}+ \mathcal{Z'}(u_{i+\frac{1}{2}}){u'}_{i+\frac{1}{2}}+\gamma(x_{i+\frac{1}{2}})\Delta x^2 + O(\Delta x^4)\right\}\\
    &-\frac{1}{\Delta x}\left\{\mathcal{M}\left(\frac{u_i+u_{i-1}}{2}\right) + \beta\left(\frac{u_i+ u_{i-1}}{2}\right)\left(\frac{u_{i-1}-u_{i}}{2}\right)^2 + O(\Delta x^3)\right\}\left\{1+ u^{(3)}_{i-\frac{1}{2}}+ \mathcal{Z'}(u_{i-\frac{1}{2}}){u'}_{i-\frac{1}{2}}+\gamma(x_{i-\frac{1}{2}})\Delta x^2 + O(\Delta x^4)\right\}{.}
\end{align*}
Note that for any continuously differentiable function $g(s)${,}
\begin{align*} 
g\left(\frac{u_i + u_{i+1}}{2}\right) =  g(u_{i+\frac{1}{2}}) + {g'}(u_{i+\frac{1}{2}})\frac{{u''}_{i+\frac{1}{2}}}{2} \left(\frac{\Delta x}{2}\right)^2 + O(\Delta x^4),\\
\left(\frac{u_{i+1}-u_{i}}{2}\right)^2 = ({u'}_{i+\frac{1}{2}})^2 \left(\frac{\Delta x}{2}\right)^2 + O(\Delta x^4).
\end{align*}
The above properties can be applied to $\mathcal{M}(s)$ and $\beta(s)$. Hence we conclude 
\begin{align*}
    [m(u_{i-1},u_{i})(1+p_{i,\bar{x}})]_{x} = \left[\mathcal{M}(u_i)(1+u^{(3)}_i-Z'(u_i){u'}_{i})\right]' + O (\Delta x^2).
\end{align*}
\end{proof}

\section{Numerical Simulation}
\label{section: numerical simulation}
In this {S}ection~\ref{section: numerical simulation}, we present numerical simulations based on the continuous time method in {S}ection~\ref{section: PPS formulation} with a practical discrete-time adaptive time stepping method.  We illustrate the benefit of using the BEM over GM in a physically relevant setting in comparison to results from laboratory experiments. Throughout {S}ection~\ref{section: numerical simulation}, we solve problem $(P)$ with the specific functions.
\begin{equation}
\label{eqn: simulation functions}
\begin{split}
    &\mathcal{M}(h) =  \frac{h^3\phi(\alpha h)}{3\phi(\alpha)} ,\, \phi(X)=\frac{3}{16X^{3}}[(1+X)^{4}(4\log (1+X)-3)+4(1+X)^{2}-1],\\
    &\mathcal{Z}_{+}(h) =-\frac{A_{H}}{h^3}, \quad \mathcal{Z}_{-}(h) =  \frac{\alpha}{\eta(1+\alpha h)}.
\end{split}
\end{equation}

This corresponds to the FSM in Ji et al. \cite{ji2019dynamics} with $\lambda = 0$. In their work, setting $\lambda = 0$ matched the experimental data better than setting $\lambda > 0$. Thus, this is a good example to demonstrate our method on. The film stabilization term $\mathcal{Z}_{+}(h)$ takes the functional form of disjoining pressure, with $A_H$ corresponding to the Hamaker constant. Increasing the value of $A_H$ stabilizes the flow. The parameter $\eta$ acts as a scaling parameter in the azimuthal curvature $\mathcal{Z}_{-}(h)$,  and decreasing its value destabilizes the flow.

% \bh{$\mathcal{Z}_{+}(h)$ can be thought of as disjoining pressure, where $A_H$ represents the Hamaker constant. Increasing the value of $A_H$ will stabilize the profile of the flow. $\eta = (\mathcal{H}/\mathcal{R})^2$ is a scaling parameter and decreasing its value will destabilize the profile of the flow.}

For each simulation, we use the functions in~\eqref{eqn: simulation functions} and dimensionless parameters $\alpha$, $\eta$, $A_{H} {>0}$ and a dimensionless initial data $h_0(x)$ on domain $[0,L]$. In {S}ection~\ref{section: comparison of numerical schemes} and {S}ection~\ref{section: adaptive time stepping}, we use dimensionless variables to compare the performance of the two numerical schemes. Whereas, in {S}ection~\ref{section: comparison with experiment}, the simulation is compared with experimental data, so the numerical results are converted back to a dimensional scale. The dimensionless parameters and the initial data are chosen to be in the range of physically meaningful values. Many times, we choose the initial data as a slightly perturbed constant state{,}
 \begin{equation*}
    h_0(x) = \bar{h}(1+0.01\sin(\pi x/L)).
\end{equation*}
The initial condition represents the profile of a flat liquid film at the onset of the instability, where $\bar{h}$ is a critical flow parameter that governs the size, spacing, and frequency of the liquid beads, consequently having a strong influence on the flow regime \cite{sadeghpour2017effects}. 
 
\subsection{Comparison of Numerical Schemes}
\label{section: comparison of numerical schemes}
In this section, we compare the simulation of BEM and GM in a physically relevant setting. We simulate BEM and GM with the functions~\eqref{eqn: simulation functions} with dimensionless parameters $\alpha = 10.6$, $\eta = 0.223227$, $A_{H} = 0.001$. We choose the initial data as
 \begin{equation*}
    h_0(x) = 1.471(1+0.01\sin(\pi x/L)), \,\, L = 24.0.
\end{equation*}

The numerical schemes presented in {S}ection~\ref{section:positivity proof discrete solutions} are continuous in time. Thus, we must discretize the time step for the practical implementation. We discretize the continuous method~\eqref{pps} using the $\theta $-weighted time-step method with $\theta = \frac{1}{2}$ (semi-implicit). This leads to the semi-implicit BEM method:
\begin{bemthod}[Semi-implicit BEM]

\begin{align}
 & \left(1+\alpha \frac{u^{k+1}_{i}+u^{k}_{i}}{2}\right) \left(\frac{u^{k+1}_{i}-u^{k}_{i}}{\Delta t}\right)+  [m(u^{k+1}_{i-1},u^{k+1}_{i})(1+p^{k+1}_{i,\bar{x}})]_{x} = 0, \label{method: BEM discrete time}\\
 & p^{k+1}_{i} = u^{k+1}_{i, \bar{x}x} - \mathcal{Z}_{+}(u^{k+1}_{i})- \mathcal{Z}_{-}(u^{k}_{i}), \\
& u_i(0) = u_0(i\Delta x), \, i = 0,1,2 \cdots N,\\
 & m(s_1,s_2) = \begin{cases}
           \mathcal{M}(s_1)  &\text{if}\,\, s_1 = s_2,\\
        (s_2-s_1)/\int_{s_1}^{s_2}\frac{1}{\mathcal{M}(s)}ds &\text{if}\,\, s_1 \ne s_2.
       \end{cases}
\end{align} 
\end{bemthod}

While other terms involving spatial differences, including $\mathcal{Z}_{+}$, are discretized implicitly, we note that $\mathcal{Z}_{-}$ is discretized explicitly. Such discretization is a well-known technique that increases the stability of a numerical method by treating a concave term and a convex term separately { \cite{barrett1998finite,eyre1998unconditionally,doi:10.1137/0730084}}. One may employ a fully implicit method, but this typically requires $\Delta t$ to be very small. We observe that the semi-implicit method is stable for larger time steps. When using the semi-implicit scheme, we accelerate the simulations by incorporating adaptive time stepping, as discussed in detail in {S}ection~\ref{section: adaptive time stepping}. We also note that one has to numerically calculate $\int_{s_1}^{s_2}\frac{1}{\mathcal{M}(s)}ds$ while evaluating $m(s_1,s_2)$. We use the Simpson's method with 2-4 grids to numerically integrate $1/\mathcal{M}(h)$ on $[u_{i-2},u_{i-1}], [u_{i-1},u_i]$, and so on. Similarly, we discretize the continuous method~\eqref{GM} using the fully implicit time-stepping scheme in \cite{ji2019dynamics}.
\begin{genthod}[Implicit GM with discrete mobility]

\begin{align}
 &\left(1+\alpha \frac{u^{k+1}_{i}+u^{k}_{i}}{2}\right) \left(\frac{u^{k+1}_{i}-u^{k}_{i}}{\Delta t}\right)+  [m(u^{k+1}_{i-1},u^{k+1}_{i})(1+p^{k+1}_{i,\bar{x}})]_{x} = 0, \label{method: GM discrete time}\\
 & p^{k+1}_{i} = u^{k+1}_{i, \bar{x}x} - \mathcal{Z}_{+}(u^{k+1}_{i})- \mathcal{Z}_{-}(u^{k+1}_{i}),\\
& u_i(0) = u_0(i\Delta x), \, i = 0,1,2 \cdots N,\\
& m(s_1,s_2) = \begin{cases}
           \mathcal{M}(s_1)  & \text{if}\,\, s_1 = s_2,\\
           \mathcal{M}\left( 0.5(s_1 + s_2)\right) & \text{if}\,\,s_1 \ne s_2.
       \end{cases}
\end{align}
\end{genthod}

We take $m(s_1,s_2) = \mathcal{M}\left( 0.5(s_1 + s_2)\right)$, which satisfies Definition~\ref{mobility discretization}. The calculation of $m(s_1,s_2)$ for GM is relatively simple since it does not require numerical integration. As mentioned before, the GM is fully implicit so $\Delta t$ needs to be well-controlled and kept small. Thus, when we compare the simulation of BEM~\eqref{method: BEM discrete time} to GM~\eqref{method: GM discrete time} in {S}ection~\ref{section: comparison of numerical schemes}-\ref{section: comparison with experiment}, we use a fixed $\Delta t$ unless the numerical method fails to converge in which case we decrease $\Delta t$ by half. In {S}ection~\ref{section: adaptive time stepping}, we show an example of BEM~\eqref{method: BEM discrete time} implemented with the adaptive time stepping algorithm (see Algorithm~\ref{algorithm: adaptive time stepping}) to demonstrate more efficient implementation. For both methods, we use Newton's method at each time step to solve discrete nonlinear equations. The Newton's method returns \texttt{True} if it successfully solves for the numerical solution at the next time step within 15 iterations; otherwise, it returns \texttt{False}. When the Newton's method fails, we decrease $\Delta t$ by 50\% and try Newton's method again. The detailed procedure of the Newton's method is written in Algorithm~\ref{algorithm: newton's method} of~\ref{asection: Algorithms used for Numerical Simulation}.

\begin{figure}[h!]
\begin{subfigure}{0.50\textwidth}
\includegraphics[width = 1.0\linewidth]{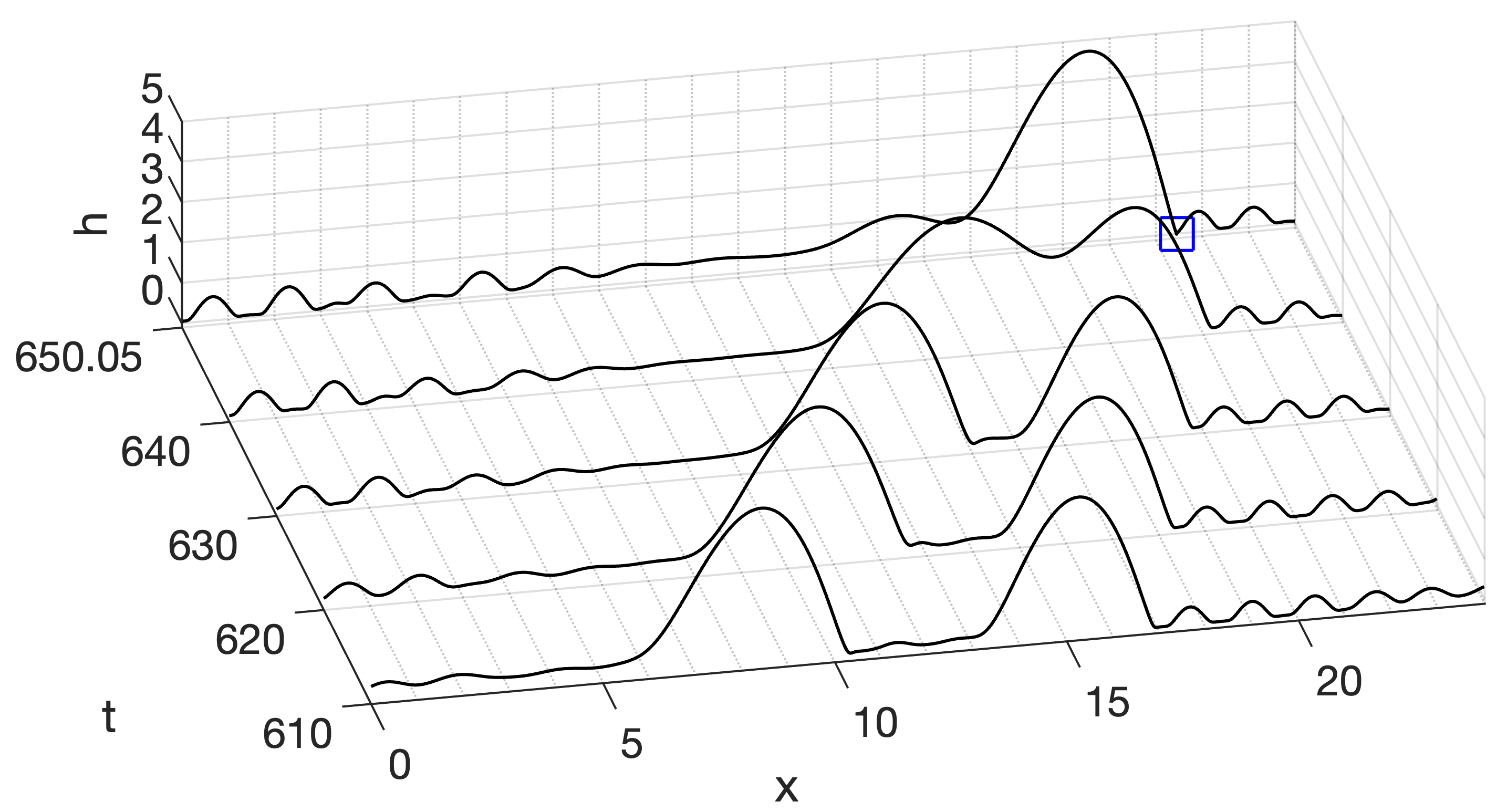}  
\caption{GM}
\end{subfigure}
\begin{subfigure}{0.47\textwidth}
\includegraphics[width = 1.0\linewidth]{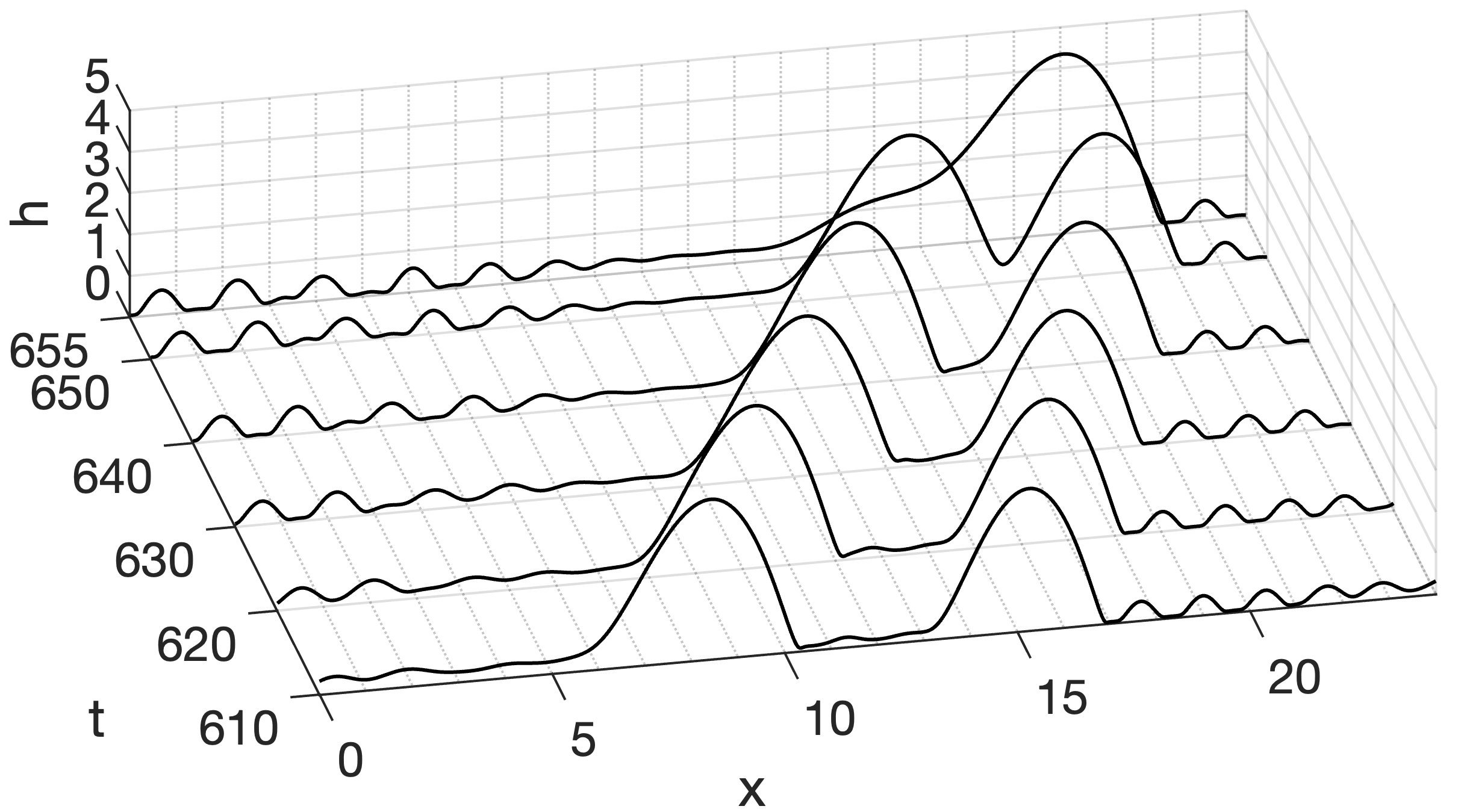}  
\caption{BEM}
\end{subfigure}
\caption{Simulation results with (a) Generic Method (GM)~\eqref{method: GM discrete time} from $t = 610$ to $t = 650.05$ and (b) Bounded Entropy Method (BEM)~\eqref{method: BEM discrete time} from $t = 610$ to $t ={ 655}$ on a coarse grid (3072 grid points on $[0, 24]$). The details of the simulation {are} described in Sec.~\ref{section: comparison of numerical schemes}. The plots illustrate the difference between the evolution profiles of traveling droplets as they merge. At $t = 640$, GM prematurely fuses two droplets while BEM does not. Because of the instability caused during the merging, GM develops negativity at $t = 650.05$, indicated by the blue square marker. The instability also causes the Newton's method to fail for GM at $t = 650$, so $\Delta t = 0.1$ is decreased by half $\Delta t = 0.05$. On the other hand, BEM can handle such {an} instability {(see $t = 655$)} and maintain the positivity of the film thickness while keeping the time step size $\Delta t = 0.1$.}
\label{fig: evolution comparison simulation}
\end{figure}

Figure~\ref{fig: evolution comparison simulation} and Figure~\ref{fig: comparison snapshot of BEM and GM} compare numerical simulations of the GM~\eqref{method: GM discrete time} and the BEM~\eqref{method: BEM discrete time} methods on a dimensionless domain $[0,24]$. In Figure~\ref{fig: evolution comparison simulation}, one observes a classic evolution of isolated droplet dynamics where the bigger droplet collides with a smaller one and merges into one droplet as the solution propagates. Figure~\ref{fig: comparison snapshot of BEM and GM} is a closeup of the results from Figure~\ref{fig: evolution comparison simulation} at the time of singularity. To generate Figure~\ref{fig: evolution comparison simulation} and Figure~\ref{fig: comparison snapshot of BEM and GM}, we simulate GM on a fine grid (6144 grid points on $[0, 24]$) until dimensionless time $t = 610$ with $\Delta t = 10^{-4}$ fixed. At this time $t = 610$, we extract the data corresponding to a coarse grid (3072 grid points on $[0, 24]$, which is twice the grid size of the fine grid) and set it as an initial condition for Figure~\ref{fig: evolution comparison simulation} and Figure~\ref{fig: comparison snapshot of BEM and GM}. From this time, we simulate BEM and GM on the coarse grid with fixed $\Delta t = 0.1$. Figure~\ref{fig: evolution comparison simulation}(a) illustrates the evolution of the simulation of GM while Figure~\ref{fig: evolution comparison simulation}(b) illustrates the evolution of the simulation of BEM. At $t = 650.05$ in Figure~\ref{fig: evolution comparison simulation}(a), one observes that the numerical solution becomes negative at one grid point in an underresolved mesh setting. Notice that Figure~\ref{fig: evolution comparison simulation}(a) has a singularity at $t = 650.05$ instead of $t = 650.0$ or $t=650.1$ despite keeping $\Delta t = 0.1$ fixed. This is because, at $t = 650$, the Newton's method for GM fails. As a consequence, the time steps $\Delta  t = 0.1$ is decreased by half, $\Delta t = 0.05$ (see Algorithm~\ref{algorithm: newton's method} in~\ref{asection: Algorithms used for Numerical Simulation}). The Newton's method succeeds after decreasing the time step by half, yet the recovered solution has a negative $h$ value. On the other hand, BEM successfully maintains positivity throughout the dynamics.

In Figure~\ref{fig: comparison snapshot of BEM and GM}, one observes the detailed profile of each simulation at the time of the numerical singularity. We continue the simulation in Figure~\ref{fig: evolution comparison simulation} until $t = 654.$ Note that we observe the numerical singularity on the coarse GM~\eqref{method: GM discrete time} simulation at $t = 650.05$ for the first time. The coarse GM simulation continues to have a negative value in contrast to the coarse BEM~\eqref{method: BEM discrete time} simulation, which stays positive. Having a singularity is critical since it often prevents further numerical simulation and provides inaccurate results. It is also unphysical because no finite time rupture is observed in the experiment. Such numerical singularities are commonly observed with the GM method in this dynamic regime of the simulation. The details of the fixed time closeup are described in the caption of Figure~\ref{fig: comparison snapshot of BEM and GM}.

\begin{figure}[h!]
\begin{subfigure}{0.60\textwidth}
\centering
\includegraphics[width = 0.8\linewidth]{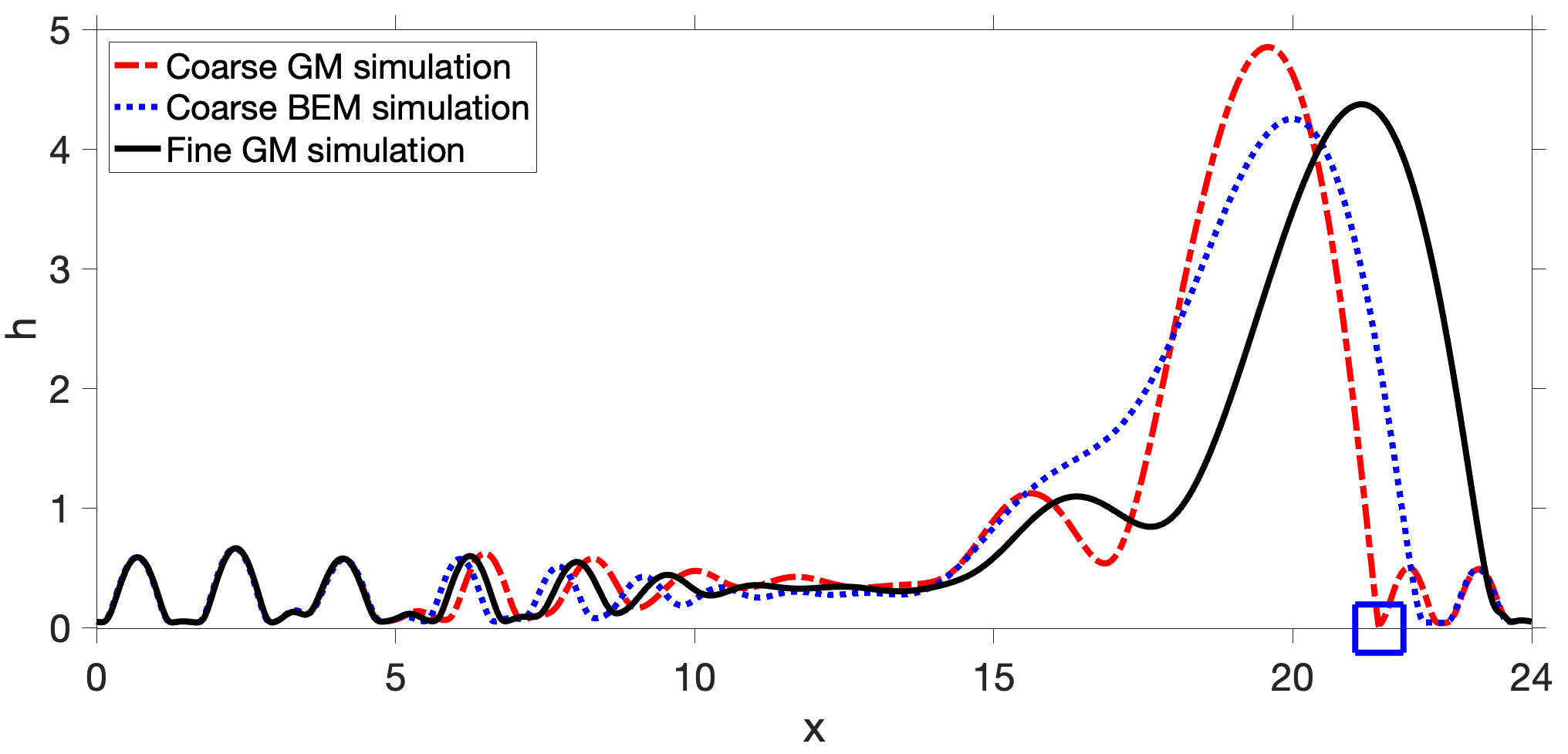}
\caption{Simulation comparison}
\end{subfigure}
\begin{subfigure}{0.37\textwidth}

\includegraphics[width = 1.0\linewidth]{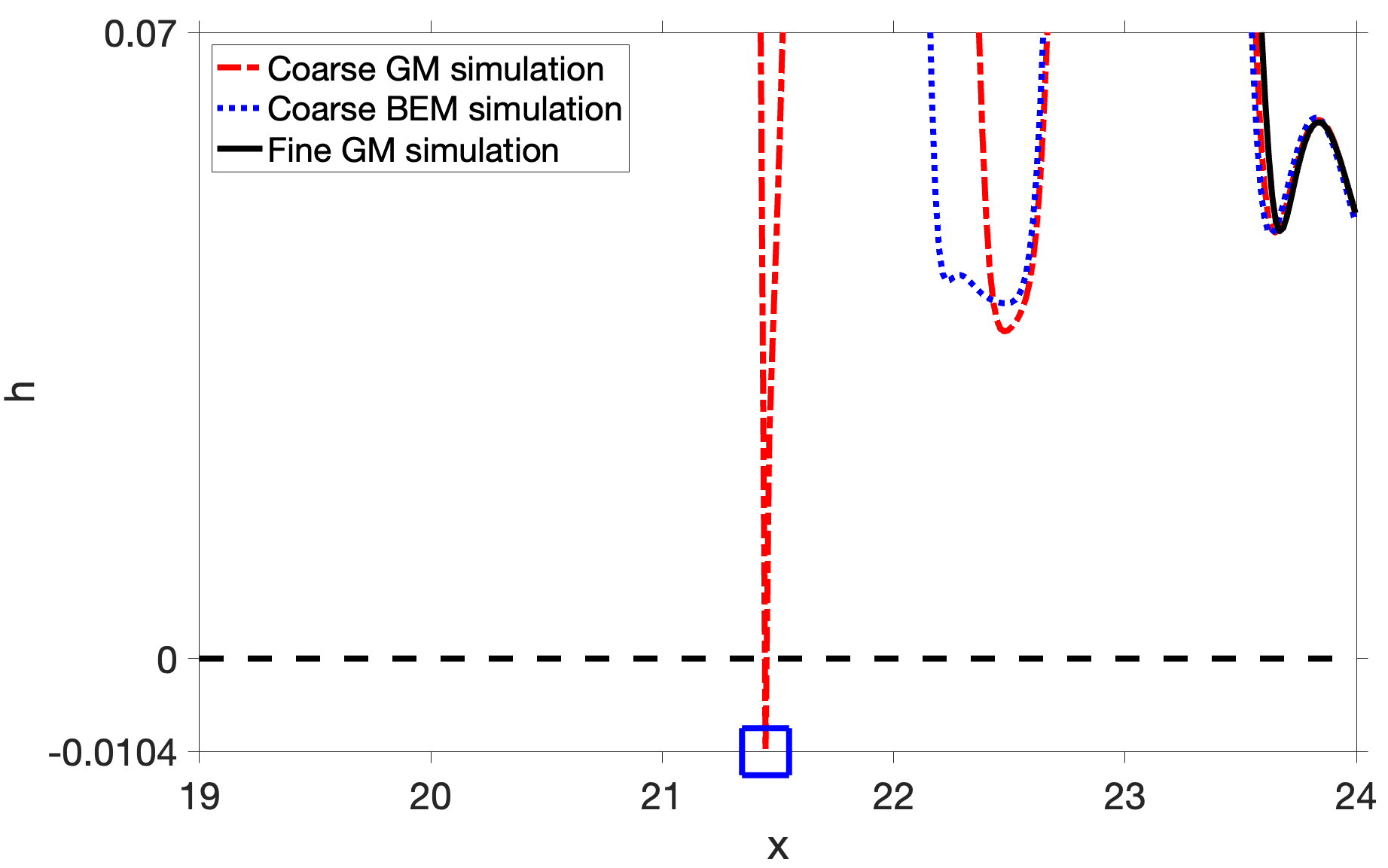} 
\caption{Simulation comparison (enlarged)}
\end{subfigure}
\caption{A closeup of a coarse grid simulation (3072 points on $[0, 24]$) around $t =  654$. The details of the simulation are described in Sec.~\ref{section: comparison of numerical schemes}. The coarse GM simulation is taken at $t = 654.45$, the coarse BEM simulation is taken at $t = 654.40$, and the fine GM simulation is taken at $t = 654.41$. Figure~\ref{fig: comparison snapshot of BEM and GM}(a) represents the full profile, and Figure~\ref{fig: comparison snapshot of BEM and GM}(b) represents the closeup profile near the singularity. Note that $h$ of the coarse GM simulation goes below the zero line indicated in dashed black at $t = 654.4500$, whereas the coarse BEM simulation does not go below the zero line at $t = 654.400$. The fine GM simulation uses twice as many grid points ($6144$ grid points on $[0, 24]$) and is captured at $t = 654.4100$.  Besides the phase shift, the coarse BEM simulation agrees better with the fine GM simulation in the sense that the average $l_2$ error ($l_2$ error = $2.0116$) across the domain is lower than the average $l_2$ error caused by coarse GM simulation ($l_2$ error = $2.5999$). The average $l_2$ error was calculated by equation~\eqref{eqn: l2 error}.} 
\label{fig: comparison snapshot of BEM and GM}
\end{figure}

One can see that the singularity affects the shape of the solution making the numerical prediction inaccurate. Let us take a closer look at the downstream and upstream profile of the droplet in Figure~\ref{fig: comparison snapshot of BEM and GM}. We see that the coarse BEM~\eqref{method: BEM discrete time} simulation has more smoothness downstream of the droplet (from $x = 23$ to $x = 24$), whereas GM~\eqref{method: GM discrete time} simulation has a finite time pinchoff (marked by a blue square). We also see that BEM's wavy pattern at the upstream matches better with the experiment than the GM's (from $x = 0$ to $x = 15$). Furthermore, the coarse BEM simulation has a lower average $l_2$ error ($l_2$ error = $2.0116$) than the error caused by the coarse GM simulation ($l_2$ error = $2.5999$) despite using different schemes. Here, we define the average $l_2$ error as
\begin{equation}
\label{eqn: l2 error}
   l_2\,\text{error} = \frac{1}{L} \sum_i (u_i -u^{*}_i)^2,
\end{equation}
where $u_{i}$ is the simulation results on the coarse grid and $u^{*}_i$ is the simulation result on the fine grid at the corresponding points of the coarse grid.

\subsection{Comparison with laboratory experiment}
\label{section: comparison with experiment}
Here we compare predictions from our method with the experimental data. In the experiment, the coating flow is created by injecting a fluid into the nozzle with an inner diameter of 0.8 mm using a programmable syringe pump. We use Rhodorsil silicone oil v50, which is a well-wetting liquid with the density $\rho$ = 963 kg/m$^3$,  kinematic viscosity $\nu$ = 50 mm$^2$/s,  and surface tension $\sigma$ = 20.8 mN/m at 20$^{\circ} C$. The corresponding capillary length $l_c$ = 1.5 mm. The fluid flows along 0.6 m-long Nylon string that is hung vertically. The radius of the Nylon string is 0.1 mm.  A high-speed camera captures the flow at a frame rate of 1000 frames/second. We estimate the measurement uncertainty in the liquid bead radius and length to be approximately $\pm$ 0.08 mm, and that in the liquid bead spacing approximately $\pm$ 0.3 mm. Further details of our experimental setup, procedure, and data analysis can be found in a previous publication \cite{sadeghpour2017effects}.

We consider two cases: the Rayleigh-Plateau case and the isolated droplet case. We do not consider the convective regime because it requires different boundary conditions. For the first case, we let the flow rate be 0.08 g/s for a fiber with a radius of 0.1 mm and a nozzle inner diameter (nozzle ID) of 0.8 mm. The experiments and corresponding numerical method both exhibit the Rayleigh-Plateau regime (see Figure~\ref{fig: RP}). For the second case, we let the flow rate be 0.006 g/s for the same fiber.  For these parameters, one observes the isolated droplet regime (see Figure~\ref{fig: IS}).

\begin{figure}[h!]
\includegraphics[width = 1.0\linewidth]{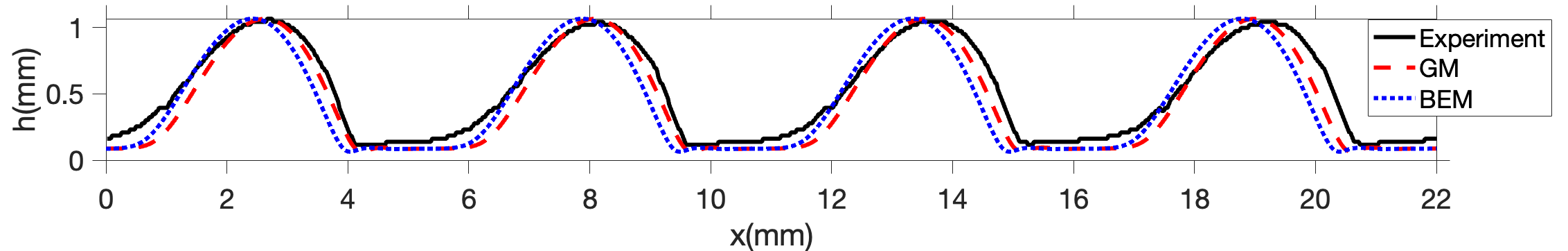}  
\caption{Comparison between laboratory experimental data and simulation data of the numerical methods. The details of the simulation and laboratory data acquisition are described in Sec.~\ref{section: comparison with experiment}. GM~\eqref{method: GM discrete time} and BEM~\eqref{method: BEM discrete time} were simulated with a fine grid ($1000$ grid points on the domain $[0, 5]$) and then shifted horizontally to match the phase. The experimental profile (the black solid line) follows the Rayleigh-Plateau regime, extracted from an experiment conducted with a flow rate of 0.08 g/s, a fiber radius of 0.1 mm, and nozzle ID of 0.8 mm.}
\label{fig: RP}
\end{figure}

\begin{figure}[h!]

\includegraphics[width = 1.0\linewidth]{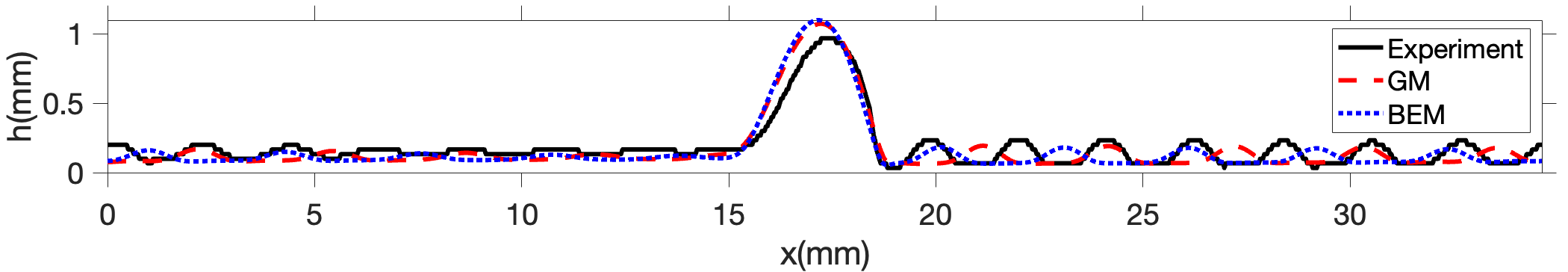}  
\caption{Comparison between laboratory experimental data and simulation data of the numerical methods. The details of the simulation and laboratory data acquisition are described in Sec.~\ref{section: comparison with experiment}. GM~\eqref{method: GM discrete time} and BEM~\eqref{method: BEM discrete time} were simulated with a relatively coarse grid ($1999$ grid points on the domain $[0, 39.338]$) and then shifted horizontally to match the phase. The experimental profile (the black solid line) follows the isolated droplet regime, extracted from an experiment conducted with a flow rate of 0.06 g/s, a fiber radius of 0.1 mm, and a nozzle ID of 0.8 mm. }
\label{fig: IS}
\end{figure}

The experimentally obtained images are processed and segmented by the built-in methods in \textsc{Matlab}, where we have incorporated the Canny method and Otsu's method. By processing high-resolution optical images and using other experimental values such as the flow rate, fiber radius, the density of the fluid $\rho$, and the kinematic viscosity $\nu$, we obtain a characteristic length scale $h_N$ and the estimated period $L$ of traveling beads. Using these values, one can calculate dimensionless parameters $\alpha$, $\eta$, and nondimensional scaling constants. We must perform this task for each experimental case since the resulting parameters are different. To generate Figure~\ref{fig: RP} and Figure~\ref{fig: IS}, we simulate GM~\eqref{method: GM discrete time} and BEM~\eqref{method: BEM discrete time} on a dimensionless domain and scale back to dimensional data to compare with the experimental data.

Figure~\ref{fig: RP} illustrates the simulation results of GM~\eqref{method: GM discrete time} and BEM~\eqref{method: BEM discrete time} compared with the experimental data of the Rayleigh-Plateau regime. We simulate GM and BEM with the functions~\eqref{eqn: simulation functions} with corresponding $\alpha$ =  5.8856 and $\eta$ = 0.2912 with a stabilizing parameter chosen to be $A_{H} = 10^{-11} $. We choose the initial data as a slightly perturbed constant state
\begin{equation*}
h_0(x) = \bar{h}(1+0.01\sin(\pi x/L)),\,\, L = 5.0, \,\,\bar{h} = 0.9568.
\end{equation*}
Note that the stabilizing parameter $A_H$ is relatively small compared to $\eta$ or $\alpha$ or the average film thickness $\bar{h}$. We simulate GM and BEM on a fine grid until dimensionless time $t = 250.006$ with adaptive time where $10^{-3} \leq \Delta t \leq 10^{-2}$. The adaptive time stepping was used to expedite the simulation process, but we made sure the $\max \Delta t$ is small enough for an accurate simulation (i.e. one results in the almost identical simulation if we keep $\Delta t = 10^{-4}$, fixed). After the simulation, we dimensionalize the data by multiplying scaling constants with respect to space and time. One can see that the three simulations match well despite the fact that both GM and BEM slightly underpredict the bead traveling speed as they go further along the $x$-direction.

Figure~\ref{fig: IS} illustrates the simulation results of GM~\eqref{method: GM discrete time} and BEM~\eqref{method: BEM discrete time} compared with the experimental data of the isolated droplet regime. We simulate GM and BEM with the functions~\eqref{eqn: simulation functions} with corresponding $\alpha =  3.092621559$ and $\eta = 0.123$ with a stabilizing parameter chosen to be $A_{H} = 4.0\times10^{-2} $. Note that the stabilizing parameter $A_{H}$ is bigger than the value we choose to simulate the Rayleigh-Plateau regime. We have simulated GM and BEM with a slightly perturbed constant state condition as the initial data, but the simulation has resulted in a dramatically different and unphysical profile from the experimental data. We expect this to be natural because the profile of the isolated droplet regime is inherently more complex than the Rayleigh-Plateau regime. We expect that there are several different steady states, and it may depend on the initial data intricately. Therefore, we extract the initial condition from the experiment and use an interpolating sine series to find the best-fitting smooth function. We enforce a periodic boundary condition by cropping the data appropriately so that the $h_0$ at $x = 0$ matches $h_0$ at $x = L$. After cropping, we use a moving average filter to smooth data even further. The code implementation details are published in a GitHub repository \cite{github}. After acquiring the initial data, we simulate GM and BEM on a fine grid until dimensionless time $t =  807.107$ for GM and $t = 827.8070$ for BEM with adaptive time where $10^{-3} \leq \Delta t \leq 10^{-2}$. The adaptive time stepping is used to expedite the simulation process again. Similar to the Rayleigh-Plateau simulation, we dimensionalize the data by multiplying scaling constants with respect to space and time. One can see that both simulations predict the width of the droplet well with slight overprediction of the height of the droplet. We note that BEM describes the pinchoff behavior downstream of the bead better (from $x = 18mm$ to $x=20 mm$)  than GM since GM is nearly flat in this region (from $x = 18mm$ to $x=20 mm$) in Figure~\ref{fig: IS}.

\RestyleAlgo{ruled}
\begin{algorithm}[H]
 \caption{Adaptive {time stepping} for BEM~\eqref{method: BEM discrete time} described in Sec.~\ref{section: adaptive time stepping}}\label{algorithm: adaptive time stepping}
\KwData{Discrete initial data $\mathbf{u}^{0}$, time step $\Delta t$, final time $t_{end}$, adaptive time tolerance  $tol_1$, the maximum number of count \texttt{countMax}}
  \KwResult{$\mathbf{u}^{k}$ at the $t_{end}$ if the simulation succeeds. Otherwise outputs $\mathbf{u}^{k}$ at the time of the simulation failure.}
    \algrule
\ttfamily
SimulateAdaptive($\mathbf{u}^{0}$,$\Delta t$,$t_{end}$):\\
set $t = 0$, bad = 0, count = 0, and $\mathbf{u}^{k} = \mathbf{u}^{0}$\;
\While{$t < t_{end}$}
{
    \eIf{\texttt{NewtonMethod($\mathbf{u}^{k}$,$\Delta t$,$tol_1$) == True}}
    {
        $t=t+\Delta t,\mathbf{u}^{k} = \mathbf{u}^{k+1}$\Comment*[r]{Update time and solution}
         $\Delta t = \Delta t * 1.01$\Comment*[r]{Increase $\Delta t$ by 1\%}
        calculate $\mathbf{e}^{k+1}, \mathbf{e}^{k}$, and  $\mathbf{LTE}(t^{k+1})$\;
        \If{$\lVert \mathbf{LTE}(t^{k+1})\rVert_{\infty} < tol_1$}
        {
            count=count+1\;
            \If{count $=$\texttt{countMax}}
            {
                $\Delta t = \Delta t * 1.2$\Comment*[r]{Increase $\Delta t$ by 20\%}
                count = 0\;
            }
        }

    }
    {
        bad = bad+1\;
        $\Delta t = \Delta t * 0.5$\Comment*[r]{Try the Newton's Method with smaller $\Delta t$}
        \If{bad $>4$}{
            exit(1)\Comment*[r]{Stop the simulation}
        }
    }
}
\label{alg: adaptive time BEM}    
\end{algorithm}

\subsection{Adaptive time stepping and computational efficiency}
\label{section: adaptive time stepping}
Adaptive time stepping can optimize the performance of the numerical method while still accurately capturing the droplet propagation. In the early stage of the computation, we expect to see a lot of change in the shape of the graph. Therefore, one wishes to keep the time step very small to capture the accurate profile of the solution. However, as the computation progress, the algorithm approaches a nearly steady state. It becomes costly to implement a small time step calculation for many iterations, while such a small step iteration does not contribute much to the change of the profile or the phase. Here we use an adaptive time stepping scheme motivated by the method in  \cite{bertozzi1994singularities,kostic2013statistical}.

The main idea is to use a dimensionless local truncation error for every time step and see if it surpasses a tolerance value that we impose. This choice of adaptive method was inspired by similar ideas in  \cite{bertozzi1994singularities,kostic2013statistical}. We define the dimensionless local truncation error using the following formula, 
\begin{equation*}
   LTE(t^{k+1})_{i} =  \biggl \vert e^{k+1}_{i}-\frac{\Delta t}{\Delta t_{old}} e^k_{i}\biggr \vert{,}
\end{equation*}
where 
\begin{equation*}
    e^{k+1}_{i} = \frac{u^{k+1}_{i}-u^{k}_{i}}{u^{k}_{i}},\,\, e^{k}_{i} = \frac{u^{k}_{i}-u^{k-1}_{i}}{u^{k-1}_{i}},\,\,\Delta t = t^{k+1}-t^{k},\,\,\Delta t_{old} = t^{k}-t^{k-1}.
\end{equation*}
The details of the entire algorithm are given by Algorithm~\ref{alg: adaptive time BEM}. Note that we store information from the previous timestep $\mathbf{u}^{k-1}$ to calculate $\mathbf{LTE}(t^{k+1})$. If one successfully calculates $\mathbf{u}^{k+1}$ with the Newton's method, we increase our time step by 1\%, calculate $\mathbf{LTE}(t^{k+1})$, and check $||\mathbf{LTE}(t^{k+1})||_{\infty} < tol_1$. If $||\mathbf{LTE}(t^{k+1})||_{\infty} < tol_1$ more than \texttt{countMax} times (in our case, we let \texttt{countMax} = $3$ throughout Sec.~\ref{section: adaptive time stepping}), we increase our time step by $20\%$. To speed up the simulation even further, one may increase the percentage to a higher value while the time step reduces by half if the Newton's method fails. If the error is bigger than $tol_1$, we proceed to calculate the next time step. In the case when the Newton's method fails, we decrease our time step by $50\%$ and try the Newton's method again. 

\subsubsection{An adaptive time stepping example without a singular behavior} 
\label{sec:nosing}
We simulate the semi-implicit BEM~\eqref{method: BEM discrete time} with the functions~\eqref{eqn: simulation functions}, and dimensionless parameters $\alpha = 5.0$, $\eta = 0.02$, $A_{H} = 10^{-5}$. We choose the initial data as
 \begin{equation*}
    h_0(x) = 0.95(1+0.01\sin(\pi x/L)), \,\, L = 1.0,
\end{equation*}
and use $100$ grid points on [0,1]. We start with initial $\Delta t = 10^{-3}$ and use Algorithm~\ref{algorithm: adaptive time stepping} to increase $\Delta t$ until $t = 1.0$ with $tol_1 = 10^{-1}$ and $\texttt{countMax} = 3$. Figure~\ref{nosing_fig} illustrates the increase of $\Delta t$ throughout the simulation when $\eta$ is relatively high and the stabilizing parameter $A_H$ is relatively high. Because the parameters are selected to simulate a relatively stable coating flow, the $||\mathbf{LTE}(t^{k+1})||_{\infty} < tol_1$ condition is satisfied whenever the Newton's method succeeds. Thus, every 3rd-time step (note that \texttt{countMax} = $3$), $\Delta t$ increases by 20\%.

\begin{figure}[h!]
%\label{nosing_fig}
\centering
\begin{subfigure}{0.49\textwidth}
\includegraphics[width = 1.0\linewidth]{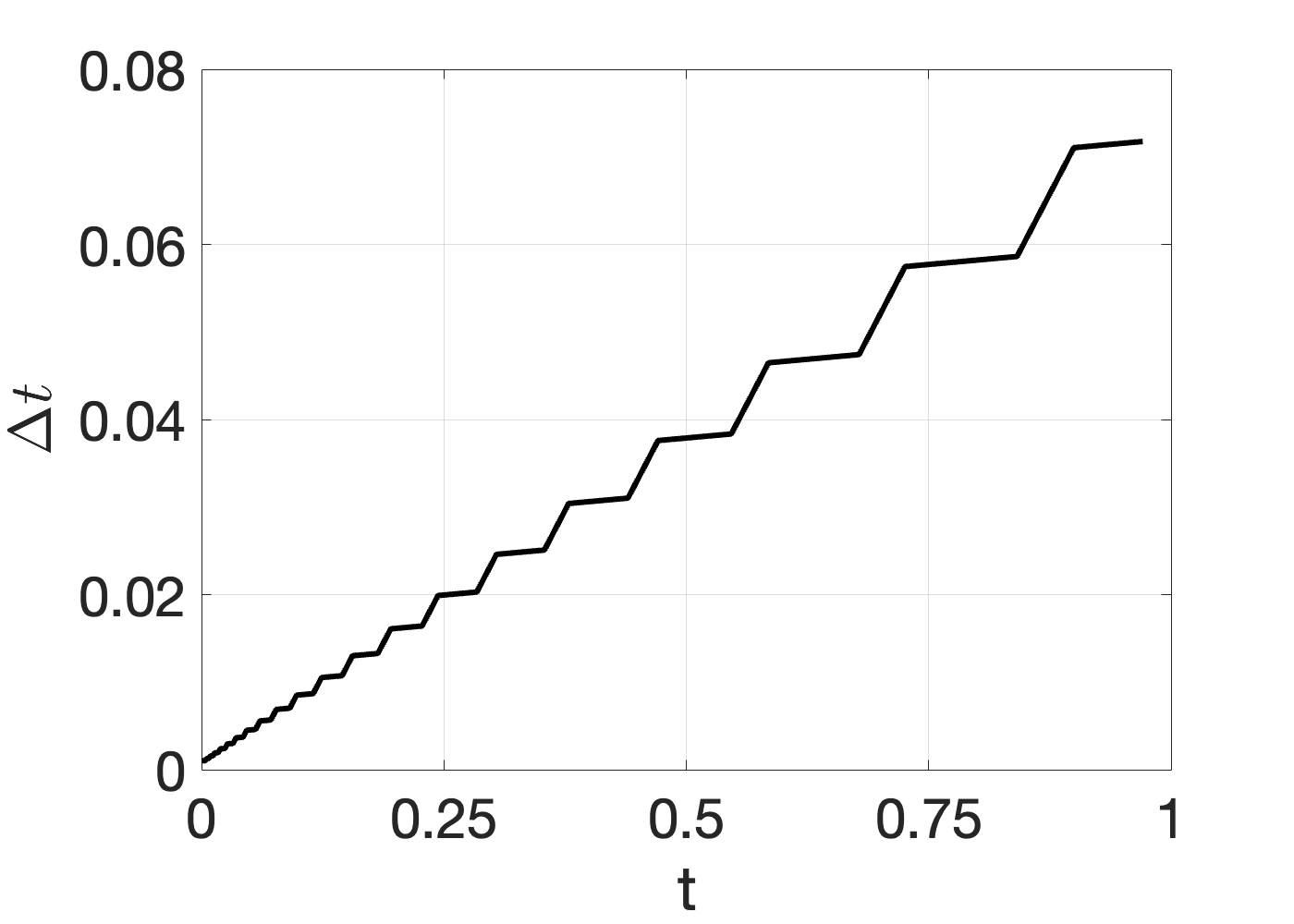} 
\caption{$\Delta t$ over simulation}
\end{subfigure}
\begin{subfigure}{0.49\textwidth}
\includegraphics[width = 1.0\linewidth]{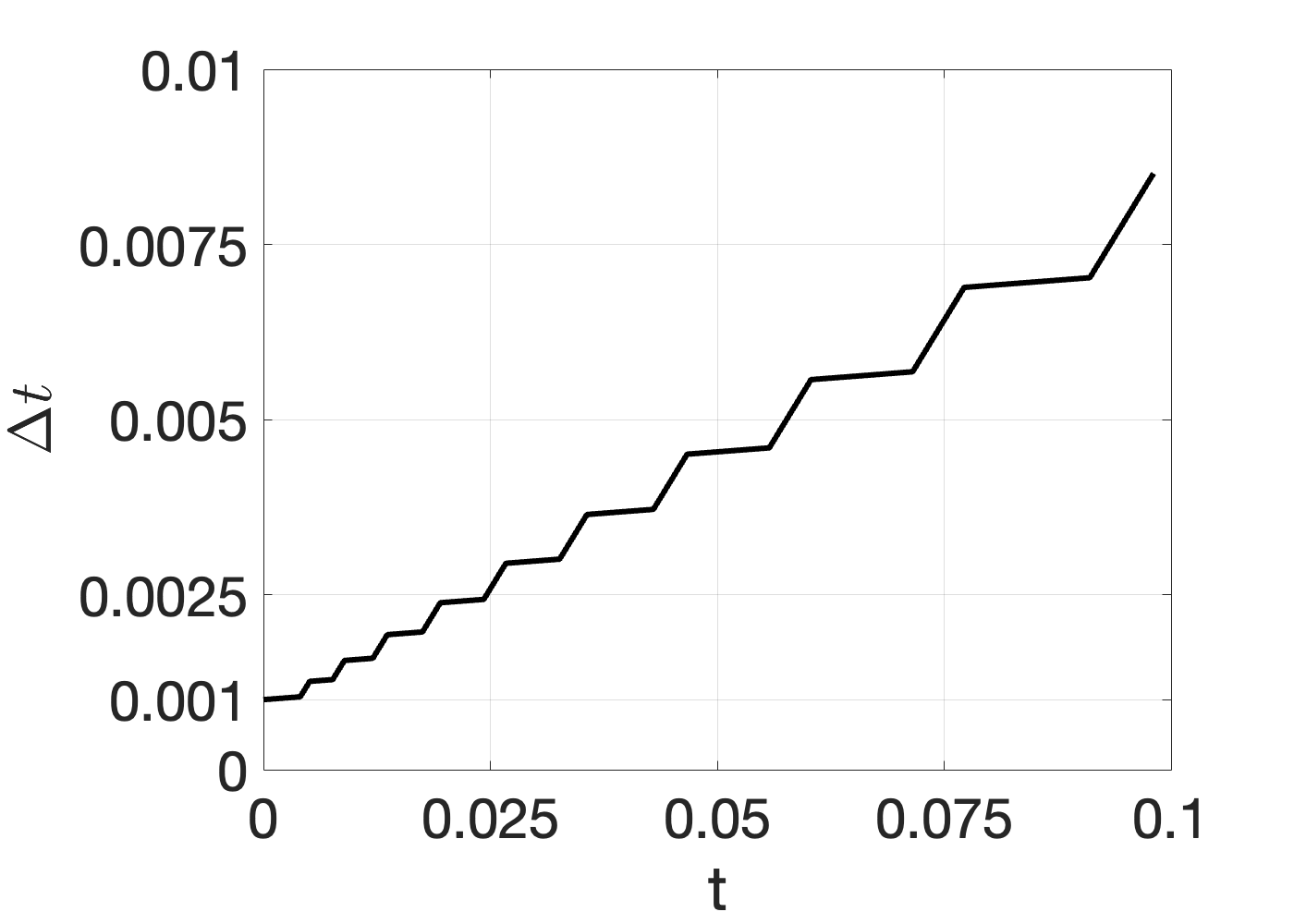} 
\caption{$\Delta t$ over simulation (enlarged)}
\end{subfigure}
\caption{Plots of $\Delta t$ for the simulation described in Sec.~\ref{sec:nosing} for $0<t<1$. The Newton's iteration always succeeds so $\Delta t$ continuously increases by $1\%$ every time while an additional increase of $20\%$ (20 times in total) occurs every 3rd time. The image on the right shows a close-up of the early time interval from $t = 0$ to $t = 0.1$. 
}
\label{nosing_fig}
\end{figure}

\subsubsection{An adaptive time stepping example with near singular behavior}
\label{sec: sing}
We simulate the semi-implicit BEM~\eqref{method: BEM discrete time} with the functions~\eqref{eqn: simulation functions}, and dimensionless parameters $\alpha = 5.0$, $\eta = 0.005$, $A_{H} = 0$. We choose the initial data as
 
 \begin{equation*}
    h_0(x) = 0.95(1+0.01\sin(\pi x/L)), \,\, L = 1.0,
\end{equation*}
and use $100$ grid points on [0,1]. We start with initial $\Delta t = 10^{-3}$ and use Algorithm~\ref{algorithm: adaptive time stepping} to increase $\Delta t$ until $t = 1.0$ with $tol_1 = 10^{-1}$ and $\texttt{countMax} = 3$ again. Since we set the stabilizing parameter $A_H = 0$ and take a lower $\eta$ value, we observe a singular behavior of the simulated flow (see Figure~\ref{fig: near sing}). Figure~\ref{sing_fig} illustrates the increase of $\Delta t$ throughout the simulation when there is a singular behavior. Unlike Figure~\ref{nosing_fig}, the $||\mathbf{LTE}(t^{k+1})||_{\infty} \geq tol_1$ from $t = 0.045228$ to $t = 0.0918907$. In this region, $\Delta t$ is increased by $1\%$ to carefully handle the transition of droplet dynamics (see Figure~\ref{fig: near sing}).

\begin{figure}[h!]
\centering
\begin{subfigure}{0.49\textwidth}
\includegraphics[width = 1.0\linewidth]{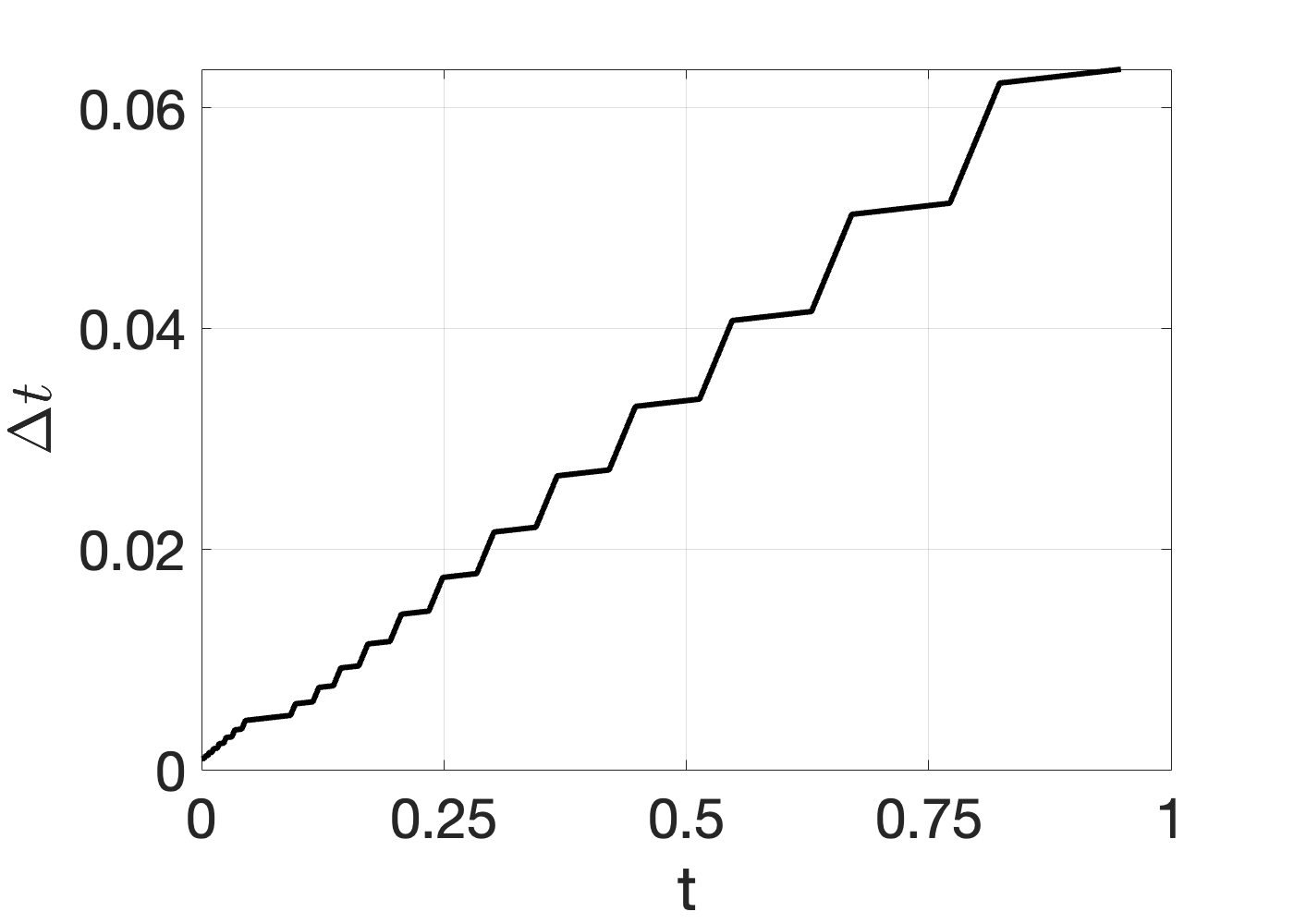} 
\caption{$\Delta t$ over simulation}
\end{subfigure}
\begin{subfigure}{0.49\textwidth}
\includegraphics[width = 1.0\linewidth]{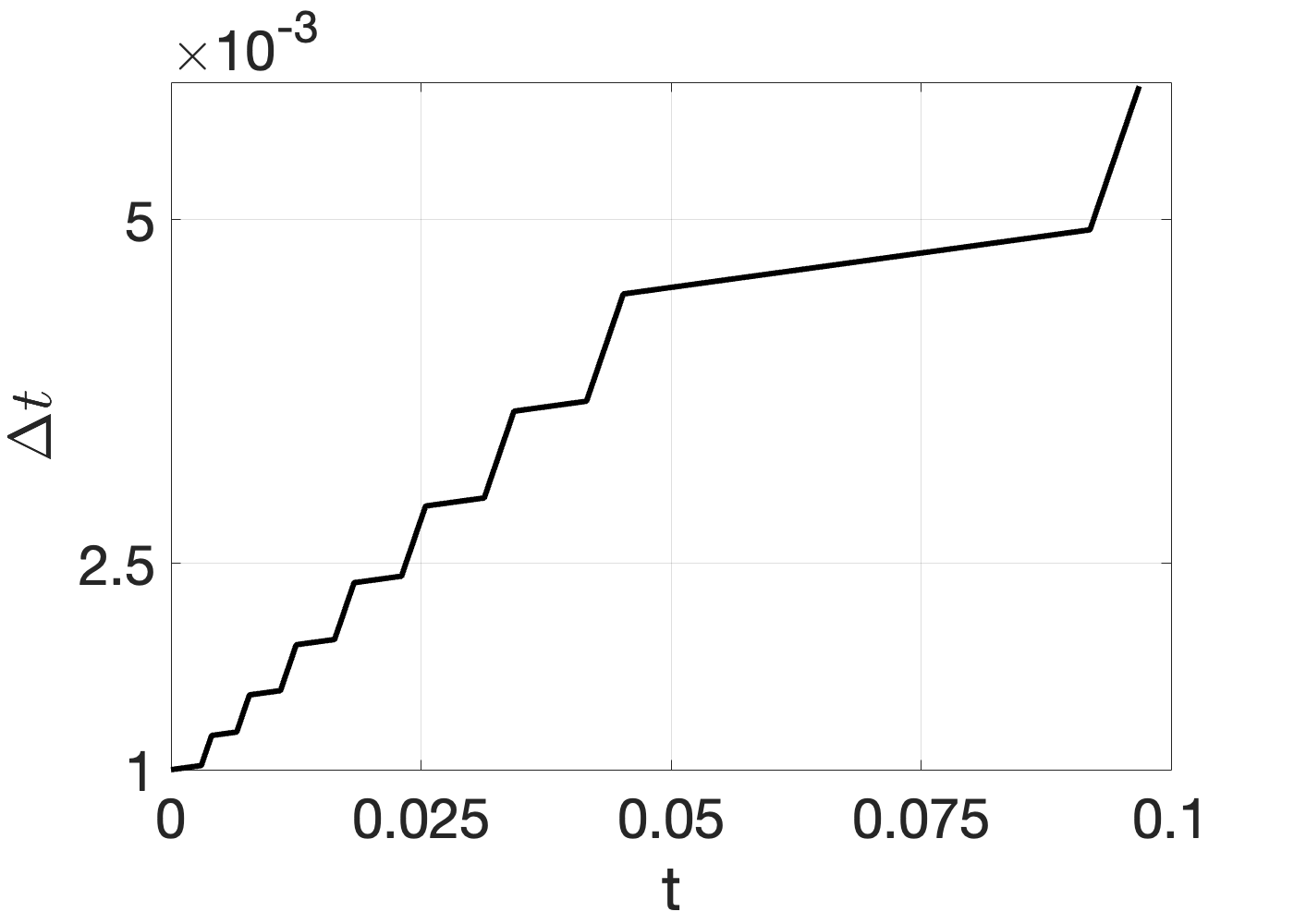} 
\caption{$\Delta t$ over simulation (enlarged)}
\end{subfigure}
\caption{Plots of $\Delta t$ for the simulation described in Sec.~\ref{sec: sing} for $0<t<1$. The Newton's iteration always succeeds so $\Delta t$ continuously increases by $1\%$ every time. However, unlike Figure~\ref{nosing_fig}, an additional $20\%$ increase occurs irregularly. In fact, from $t = 0.045228$ to $t = 0.0918907$, $\Delta t$ does not increase. The image on the right shows a close-up of the early time interval from $t = 0$ to $t = 0.1$. 
}
\label{sing_fig}
\end{figure}
\begin{figure}[h!]
    \centering
    \includegraphics[width = 0.65\linewidth]{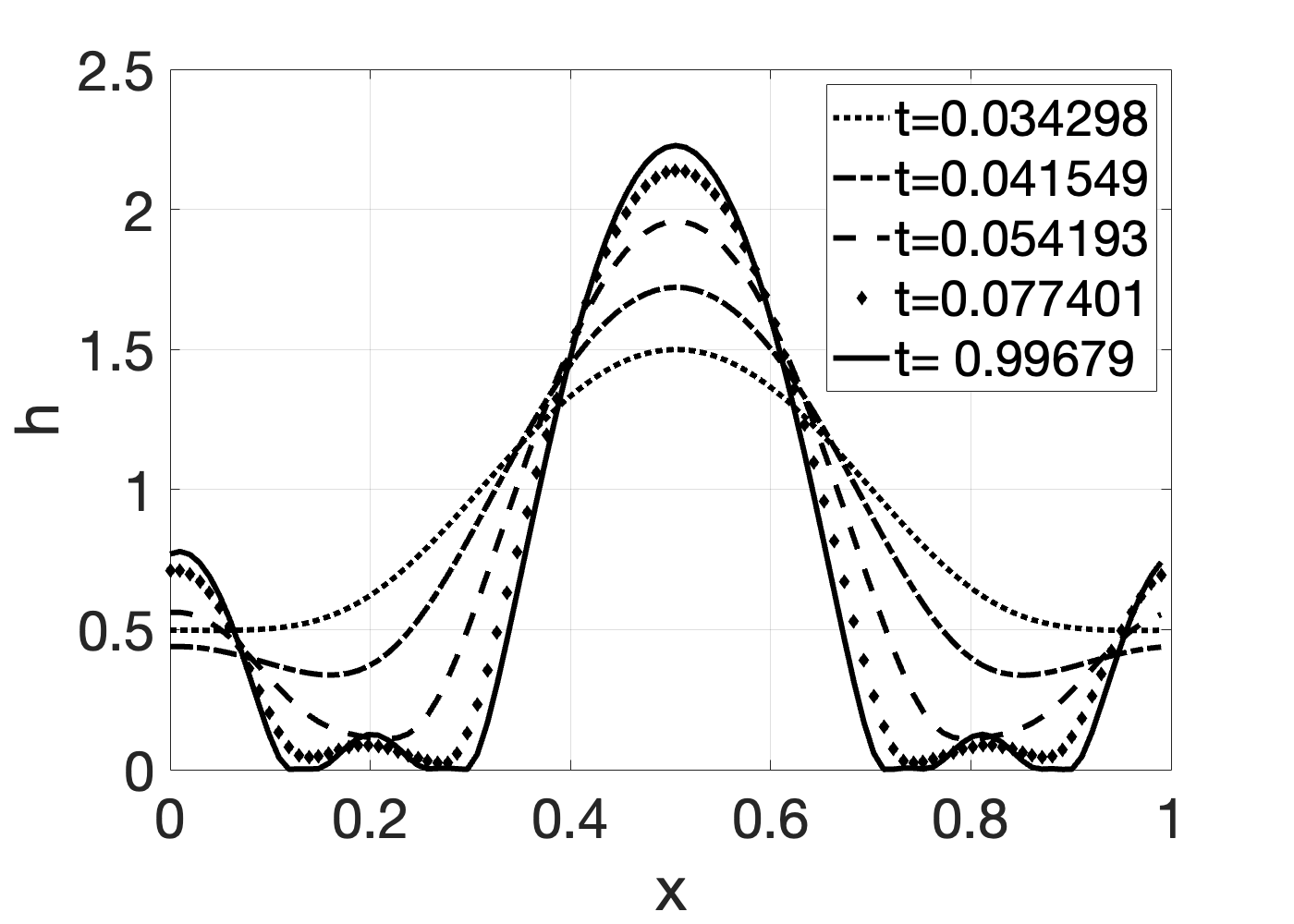}
    \caption{Evolution of a flow with a singular behavior described in Sec.~\ref{sec: sing}. All of the plots have $h \geq 6.0942\times 10^{-4}$.}
    \label{fig: near sing}
\end{figure}

\subsubsection{Computational efficiency and accuracy of the adaptive time stepping}
\label{section: computational efficiency}
In this section, we demonstrate the computational efficiency of our method BEM over GM. We simulate semi-implicit BEM~\eqref{method: BEM discrete time} and implicit GM~\eqref{method: GM discrete time} with the functions~\eqref{eqn: simulation functions}, and dimensionless parameters $\alpha = 5.0$, $\eta = 0.005$, $A_{H} = 0$. We choose the initial data as
 \begin{equation*}
    h_0(x) = 0.45(1+0.01\sin(\pi x/L)), \,\, L = 1.0,
\end{equation*}
and record the CPU time of each method on three different grid sizes. Note that this is a similar setting as the simulation run in Sec.~\ref{sec: sing}. When we use the fixed time stepping (see Algorithm~\ref{algorithm: fixed time stepping} in~\ref{asection: Algorithms used for Numerical Simulation}) for BEM and GM, we let $\Delta t = 10^{-3}$. When we use the adaptive time stepping, which is only used for BEM, we use Algorithm~\ref{algorithm: adaptive time stepping} with initial $\Delta t = 10^{-3}$, $tol_1 = 10^{-3}$, and $\texttt{countMax} = 3$. Each GM simulation is run until the numerical solution fails to preserve positivity, resulting in different termination times. On the other hand, each BEM always preserves the positivity of the numerical solution regardless of using any time stepping method so that it can be run until any time. For a fair CPU time comparison, we run BEM until GM fails with the respective grid sizes. By examining Table 1, one may notice the computational benefits of using adaptive time stepping with increased grid points.
 
\begin{table}[h!]
\centering
\begin{tabular}{ |c|c|c|c|} 
 \hline
  & Time stepping & Positivity  & CPU time\\ 
  \hline \hline
 
 GM with $\Delta x = 0.01 $ & Fixed & Fails at $t = 0.299$&0.286$s$ until $t = 0.299$ \\
 
 BEM   $\Delta x = 0.01 $& Fixed & Success   &  0.374 $s$ until $t = 0.299$ \\
 
BEM   $\Delta x = 0.01 $& Adaptive & Success  &  0.317 $s$ until $t = 0.299$ \\%0.317 
  \hline
  
 GM with $\Delta x = 0.005$ & Fixed & Fails at $t =  1.09594$ & 0.602$s$ until $t =1.09594$\\ 
 
 BEM with $\Delta x = 0.005$ & Fixed & Success  & 1.08$s$ until $t =1.096$\\ 
 
 BEM with $\Delta x = 0.005$ & Adaptive & Success   & 0.412$s$ until $t =1.09678$\\ 
 
   \hline
  GM with $\Delta x = 0.0025$ & Fixed & Fails at $t =  3.4765$ & 2.959$s$ until $t =3.4765$\\
 
    BEM with $\Delta x = 0.0025$ & Fixed & Success  &4.727$s$ until $t = 3.477 $\\ 
    
    BEM with $\Delta x =  0.0025$ & Adaptive & Success  & 0.724$s$ until $t = 3.51201$\\ 
 \hline
\end{tabular}
\caption{Computational cost comparison of BEM and GM for examples discussed in Sec.~\ref{section: computational efficiency}.}
\end{table}

\section{Conclusion}
\label{section: conclusion}
In  this  paper,  we introduce a positivity-preserving finite difference method for the problem fiber-coating a vertical cylindrical fiber. While the current {state of the art} method (GM) achieves close agreement with experiments and successfully captures regime transitions, it struggles to match the flow profiles as the film thickness becomes small. In particular, the GM needs significant grid refinement to resolve very thin films without a numerical singularity. We prove that our BEM preserves positivity given $\mathcal{M}(h) = O(h^n)$ for $n \geq 2$ and furthermore that there exists a lower bound independent of grid size given an a posteriori Lipschitz bound on the solution (something that is always observed in experiments). By constructing a generalized entropy estimate, we extend the idea of positivity-preserving methods for basic lubrication equations to the problem involving cylindrical geometry, gravity, and nonlinear pressure. This technique has promise for thin liquid film equations with complex geometry, advection effect, and other surface tension effects. 

There are a number of directions one can pursue from this work. One obvious direction is to prove the convergence of the BEM. Such work would benefit from additional regularity and positivity results for the continuum PDE. Another direction is to generalize the method to the fully 2D fiber coating problem e.g. using ADI methods such as \cite{witelski2003adi} or to consider more general geometries as in \cite{greer2006fourth}. Finally, it would be interesting to consider other types of boundary conditions since the experiment is not periodic in space. The boundary conditions on an inlet and an outlet of the flow can change if other models are considered, such as one that includes a nozzle geometry \cite{ji2020modelling} or a thermal effect \cite{ji2021thermally}.

\section*{Acknowledgments}
This research is supported by the Simons Foundation Math + X Investigator Award number 510776. H.J. acknowledges support from {NSF grant DMS 2309774 and} NCSU FRPD Program.  B.K. acknowledges support from NSF grant DGE-1829071.

%%Vancouver style references.
%\bibliographystyle{alpha}
\bibliographystyle{model1-num-names}
\bibliography{first_draft_ref.bib}

\appendix
\section{Algorithms used for Numerical Simulation}
\label{asection: Algorithms used for Numerical Simulation}
\subsection{The Newton's Method and the fixed time step algorithm}

\RestyleAlgo{ruled}
\begin{algorithm}[H]
 \caption{The Newton's method for BEM}
 \label{algorithm: newton's method}
 \KwData{numerical solution $\mathbf{u}^{k}$, the current time step $\Delta t$, and tolerance $tol$ for the convergence success criteria.}
  \KwResult{\texttt{True} or \texttt{False} depending on whether the method succeed or fail.}
    \algrule
\ttfamily
NewtonMethod($\mathbf{u}^{k}$,$\Delta t$,$tol$):\\
 $\mathbf{u}^{k+1} = \mathbf{u}^{k}$\Comment*[r]{Initial guess for the Newton's method}
 \For{$i=0$ \KwTo $15$}{
    $\mathbf{f}(\mathbf{u}^{k}) =$ the left side of the equality of {e}quation $\eqref{method: BEM discrete time}$\Comment*[r]{Use~\eqref{method: GM discrete time} for GM}
    $\mathbf{u}^{k+1} = \mathbf{u}^{k} - (\nabla \mathbf{f}(\mathbf{u}^{k}))^{-1} \mathbf{f}(\mathbf{u}^{k})$;\\
    \If{$\lVert \mathbf{f}(\mathbf{u}^{k})\rVert_{\infty} < tol/10$}
    {
        break;
    }
    
  }
    \eIf{$\lVert \mathbf{f}(\mathbf{u}^{k})\rVert_{\infty} < tol$}
    {
       \Return True;
    }{
        \Return False;
    }
\end{algorithm}
The Newton's algorithm is specifically written for BEM~\eqref{method: BEM discrete time}, but setting $\mathbf{f}(\mathbf{u}^{k})$ as the left side expression of the equality of {e}quation (10) in Algorithm~\ref{algorithm: newton's method} results in the algorithm for GM~\eqref{method: GM discrete time}. The function \texttt{NewtonMethod} has the input of the numerical solution at $k$th time step $\mathbf{u}^{k}$, the current time step $\Delta t$, and the tolerance value $tol$ which determines the success or failure of the Newton's iteration. \texttt{NewtonMethod} returns \texttt{True} if $\lVert \mathbf{f}(\mathbf{u}^{k})\rVert_{\infty} < tol$ after the for loop and updates the numerical solution by setting $\mathbf{u}^{k} = \mathbf{u}^{k+1}$. \texttt{NewtonMethod} gives a chance of 15 iteration, but in practice, we see that the method satisfies $\lVert \mathbf{f}(\mathbf{u}^{k})\rVert_{\infty} < tol/10$ within 3-4 iteration. When $\lVert \mathbf{f}(\mathbf{u}^{k})\rVert_{\infty} \geq tol$, \texttt{NewtonMethod} returns \texttt{False}.

\RestyleAlgo{ruled}
\begin{algorithm}[H]
 \caption{Simulation with regular time stepping}
 \label{algorithm: fixed time stepping}
 \KwData{a discrete initial data $\mathbf{u}^{0}$, the time step $\Delta t$, the end time $t_{end}$,  and tolerance $tol$ for the convergence success criteria.}
  \KwResult{$\mathbf{u}^{k}$ at the $t_{end}$ if the simulation succeeds. Otherwise outputs $\mathbf{u}^{k}$ at the time of the simulation failure.}
    \algrule
\ttfamily
Simulate($\mathbf{u}^{0}$,$\Delta t$,$t_{end}$):\\
set $t = 0$, bad = 0, and $\mathbf{u}^{k} = \mathbf{u}^{0}$\;
\While{$t < t_{end}$}
{
    \eIf{\texttt{NewtonMethod($\mathbf{u}^{k}$, $\Delta t$, $tol$) == True}}
    {
        $t=t+\Delta t$ \Comment*[r]{Update time}
         $\mathbf{u}^{k} = \mathbf{u}^{k+1}$\Comment*[r]{Update the numerical solution}
         bad = 0\;
    }
    {
        bad = bad+1\;
        $\Delta t = \Delta t * 0.5$\Comment*[r]{Try the Newton's Method with smaller $\Delta t$}
        \If{bad $>4$}{
            exit(1)\Comment*[r]{Stop the simulation}
        }
    }
}
\end{algorithm}

In the case when \texttt{NewtonMethod} returns \texttt{False}, we decrease $\Delta t$ by 50\% and try \texttt{NewtonMethod} again with the same $\mathbf{u}^{k}$ and $tol$ (see Algorithm~\ref{algorithm: newton's method} and Algorithm~\ref{algorithm: adaptive time stepping}). Below is the algorithm using regular time step which is used to generate Figure~\ref{fig: evolution comparison simulation} and Figure~\ref{fig: comparison snapshot of BEM and GM}. Notice that $\Delta t$ is only decreased when \texttt{NewtonMethod} returns \texttt{False}. If \texttt{NewtonMethod} fails more than 4 consecutive times, we completely stop the simulation. 

\end{document}